\documentclass[11pt]{amsart}
\usepackage[run=2]{auto-pst-pdf}
\usepackage{amsmath}
\usepackage{amssymb}
\usepackage{amscd}
\usepackage{url}
\usepackage{pstricks}
\usepackage{pstricks-add}
\usepackage{pst-node}
\usepackage{pst-coil}
\usepackage{graphicx}

\newcommand{\sawfrac}[2]{\genfrac{\langle}{\rangle}{}{1}{#1}{#2}}

\newcommand{\sawtooth}[1]{\left\langle #1\right\rangle}

\DeclareMathOperator{\lk}{lk}

\def\R{{\mathbb R}}
\def\C{{\mathbb C}}
\def\Z{{\mathbb Z}}

\def\A{{\mathcal{A}}}
\def\V{{\mathcal{V}}}
\def\E{{\mathcal{E}}}
\def\ol#1{\overline{#1}}
\def\wt#1{\widetilde{#1}}
\def\s{\mathbf{s}}
\def\Rs{\mathbf{R}}
\def\os{\overline{\sigma}}
\def\sm{\setminus}
\def\mp{m_{\textrm{p}}}
\def\mq{m_{\textrm{q}}}
\def\Mp{M_{\textrm{p}}}
\def\Mq{M_{\textrm{q}}}
\def\Sl{S_{\textrm{link}}}
\def\Sv{S_{\textrm{leaf}}}
\def\Sn{S_{\textrm{node}}}
\def\Sa{S_{\textrm{arr}}}
\def\Se{S_{\textrm{edge}}}
\def\ap{\alpha_{\textrm{p}}}
\def\aq{\alpha_{\textrm{q}}}
\def\bp{\beta_{\textrm{p}}}

\def\spp{s_{\textrm{p}}}
\def\sqq{s_{\textrm{q}}}

\def\p{\partial}
\def\cp{c_{\textrm{p}}}
\def\Gp{\Gamma_{\textrm{p}}}
\def\Gq{\Gamma_{\textrm{q}}}
\def\vn{v^{\#}}

\def\dvn{d_v^{\#}}

\theoremstyle{definition}
\newtheorem{example}[equation]{Example}
\newtheorem{definition}[equation]{Definition}

\theoremstyle{plain}
\newtheorem{assumption}[equation]{Assumption}
\newtheorem{lemma}[equation]{Lemma}
\newtheorem{proposition}[equation]{Proposition}
\newtheorem{theorem}[equation]{Theorem}
\newtheorem*{mtheorem}{Theorem}

\theoremstyle{remark}
\newtheorem*{acknowledgements}{Acknowledgements}
\newtheorem{remark}[equation]{Remark}
\numberwithin{equation}{subsection}
\makeatletter
\@namedef{subjclassname@2010}{\textup{2010} Mathematics Subject Classification}
\makeatother
\title[Signatures of graph links]{The average signature of graph links}
\author{Maciej Borodzik}
\address{Institute of Mathematics, University of Warsaw, ul. Banacha 2,
02-097 Warsaw, Poland}
\email{mcboro@mimuw.edu.pl}
\author{Jadwiga Sosnowska}
\address{Institute of Mathematics, University of Warsaw, ul. Banacha 2,
02-097 Warsaw, Poland}
\email{js277891@students.mimuw.edu.pl}

\thanks{Both authors were supported by  Polish OPUS grant No 2012/05/B/ST1/03195}

\date{\today}

\begin{document}
\begin{abstract}
We compute the average Tristram---Levine signature of any graph link with positive weights in a three sphere.
The main tools are Neumann's algorithm
for computing the equivariant signatures of graph links and the Reciprocity Law for Dedekind sums.
\end{abstract}
\maketitle

\section{Introduction}
\subsection{Statement of results}

Let $L\subset S^3$ be a link. If $S$ is a Seifert matrix for $L$, the \emph{Tristram--Levine} signature of $L$ is the piecewise constant function
from the unit circle in $\C$ to $\Z$ given by
\begin{equation}
\sigma_L(t)=\textrm{signature of the hermitian form } (1-t)S+(1-\ol{t})S^T.\label{eq:signature}
\end{equation}
The Tristram--Levine signature does not depend on the choice of the Seifert matrix. We consider also the \emph{average signature} of $L$ defined as
\begin{equation}\label{eq:average}
\os(L)=\int_0^1 \sigma_L(e^{2\pi ix})dx.
\end{equation}

The main result of the present paper is Theorem~\ref{thm:main}, which we now state.
\begin{mtheorem}
Let $L$ be a graph link in $S^3$, which is neither the unknot nor the Hopf link. Let $\Gamma$ be an underlying graph. Suppose it is almost minimal
(see Section~\ref{ss:AMD}) and let $S(\Gamma)$ be
as defined in Section~\ref{ss:sgamma}. Then
\[\os(L)=-\frac13 S(\Gamma).\]
\end{mtheorem}
Below we shall review the theory of graph links, we refer to \cite{EN} for a good introduction.

The most important message of this theorem is that $S(\Gamma)$ is a very simple function. It involves computing greatest common divisors of multiplicities
associated to some edges, beside this it depends rationally on the weights and multiplicities of the graph.
In fact, it is very easy to give a formula for $\os(L)$ for
a graph link involving Dedekind sums (it is a sum of entries from \eqref{eq:integralsplice} over splice components), but such a formula is in general
extremely hard to apply and does not give any insight into the value of $\os(L)$.

\subsection{Motivation}\label{ss:motivation}
Let us put the main result in a broader context. To start with, assume that $K\subset S^3$ is a knot. Historically, the first motivation to study the average
signature $\os(K)$ comes from signatures of branched covers, see \cite{GLM,Vi}. Let $F$ be a Seifert surface for $K\subset S^3=\p B^4$
and let us push its interior into the four ball $B^4$. Let $m>0$ be an integer and let $N_m$ be an $m$-fold cyclic
branched cover of $B^4$ along $F$. This is a four manifold.
Its signature $\tau_m:=\sigma(N_m)$ turns out to depend only on $K$ and $m$, more precisely by \cite{Vi}
\[\tau_m=\sum_{\xi\colon \xi^m=1}\sigma_K(\xi).\]
In particular, $\os(K)=\lim_{m\to\infty}\frac{1}{m}\tau_m$, that is $\os(K)$ tells us about the asymptotic behavior of $\tau_m$. This is of a special interest
in singularity theory. If $K$ is an algebraic knot, that is, a knot resulting by intersecting a zero set of a polynomial $f\colon \C^2\to\C$ with a small sphere
around $0\in\C^2$, then $N_m$ is a link of the surface singularity $\{(x,y,z)\in\C^3\colon f(x,y)+z^m=0\}$. Singularities of that type are of special interest,
we refer to \cite{Nem2} for more details.

The main problem while dealing with $\os(K)$ is that this quantity is often very hard to compute and the results do not have to be especially nice. For example,
consider the pretzel knot $P(p,q,r)$ (suppose for simplicity of the discussion that $p,q,r>0$). Its Alexander polynomial $\Delta$ satisfies
\[4\Delta(t)=t(pq+qr+pr+1)-2(pq+qr+pr-1)+(pq+qr+pr+1)t^{-1}.\]
Writing $A=pq+qr+pr$ we obtain that $\Delta$ has two roots $z_{\pm}:=\frac{A-1\pm 2i\sqrt{A}}{A+1}$. They are on the unit circle and let $\theta_\pm\in[0,1)$
be such that $e^{2\pi i\theta_{\pm}}=z_{\pm}$. Since the Tristram--Levine signature has jumps precisely at the roots of the Alexander polynomial and these
roots are simple, we infer that $\sigma_K(e^{2\pi i\theta})=0$ if $\theta\in[0,\theta_+)\cup(\theta_-,1]$ and $\sigma_K(e^{2\pi i\theta})=-2$ if
$\theta\in(\theta_+,\theta_-)$. Hence $\os(K)=\theta_--\theta_+$. This quantity, in general, is a rather complicated
analytic function of $A$. For pretzel knots with more than three strands, the formulae can be much more involved.

In \cite{KM} Kirby and Melvin showed that the average signature of a torus knot $T_{p,q}$ is equal to $-\frac13(p-\frac{1}{p})(q-\frac{1}{q})$.
By \cite[Theorem 2]{Li}, this immediately gives a closed formula for $\os(K)$ for any iterated torus knot. This result was then reproved independently
by \cite{Nem2,Bo}. As the computation of $\os(T_{p,q})$ from the formula for the Tristram--Levine signature of the torus knot (as in \cite{Li})
uses a priori Dedekind sums,  any proof of the Kirby--Melvin theorem either involves the Reciprocity Law for
Dedekind sums; or it gives another proof of the Reciprocity Law.

A new interest in computing $\os(K)$ is motivated by pioneering works of Cochran, Orr and Teichner \cite{COT1,COT2}.
The quantity $\os(K)$ turns out to be the von Neumann $\rho$--invariant (also known as the $L^2$--signature)
associated with the representation $\pi_1(S^3\sm K)\to\Z$ given by abelianization. We refer to \cite{COT1,COT2} for background on von Neumann
$\rho$--invariants for knot complements. In general, $\rho$--invariants obstruct sliceness and give insight into the structure of the topological concordance group.
We refer to \cite{Co,CHL,D1} for some exemplary applications; in fact, the literature on the subject is now very vast. We also point out that the
average signature of links with pairwise linking number~$0$ was studied to detect sliceness of some knots, see \cite{CHL,D1,D2}.

Another recent motivation for studying $\os(L)$, this time for algebraic links, is its relation to singularity theory.
In \cite{Bo} the average signature of an algebraic knot was related to an invariant of the singular point, called the $M$--number
(see \cite{Orev02} for the definition).
Using this relation, a bound for $M$--numbers under a deformation of cuspidal singular points was obtained \cite{Bo2}.
We expect a similar relation to hold for general, that is not necessarily cuspidal, singular points. This would allow
to extend results from \cite{Bo2} to more general classes of deformations. Computing the average signature of an algebraic link is the first
step towards establishing such a relation.

In the present article we deal with graph links in $S^3$. Graph links were introduced by Eisenbud and Neumann \cite{EN} as a generalization of
iterated torus knots. Any graph link can be combinatorially encoded in a graph (hence the name).
We explain this relation in Section~\ref{ss:ggl} and refer to \cite{EN} for more details.
Any algebraic link is a graph link in $S^3$. Therefore, our result gives the average Tristram--Levine signature for all algebraic links in $S^3$.

\subsection{Structure of the article}

The structure of the article is the following. After an overview of the necessary background on graph links and Dedekind sums, we provide
in Theorem~\ref{thm:main} an algorithm for computing the average signature of a graph link.
Then we prove this theorem in Section~\ref{s:proofofmain}. The idea is as follows. By an algorithm of Neumann \cite{Neu1},
one can write down the average signature of any splice component, see \eqref{eq:integralsplice}. Then we use
splice additivity of signatures, Lemma~\ref{lem:spladds}, to obtain a general formula for the average signature.
Unfortunately, the result involves many
Dedekind sums.

Using the Reciprocity Law we will show in Section~\ref{s:proofofmain}, that all the Dedekind sums that appear, can be simplified.
After somewhat lengthy, but rather straightforward computations, we obtain the desired result.

\begin{acknowledgements}
We would like to express thanks to Chris Davis for his interest in our work and for valuable comments and to Andrew Ranicki for his suggestions and help
during the
preparation of the manuscript. We are also grateful to the referee for helpful comments.
The first author is grateful to Indiana University for hospitality.
\end{acknowledgements}

\section{Graph links}
\subsection{Review of the theory of graph links}\label{ss:graph}
We begin with explaining in detail, what is a splice graph.
We refer to \cite{EN} or \cite{NW} for
a more detailed exposition. Later, in Section~\ref{ss:ggl} we shall explain how a graph gives rise to a graph link.

Throughout the paper a \emph{splice graph} $\Gamma$ will denote a collection $(\V,\A,\E)$ of the \emph{ordinary vertices} $\V$,
the \emph{arrowhead vertices} $\A$ and the \emph{edges} $\E$. For an ordinary vertex $v\in\V$, we denote by $\nu(v)$ its \emph{valency}, that
is the number of edges entering $v$. Vertices with valency $3$ or more are called \emph{nodes}, those with valency $1$ are called \emph{leaves}.
We assume that there are no vertices of valency $2$ and that there is at least one node. The valency of an arrowhead vertex is always $1$.
We will assume that $\Gamma$ is a tree, that is, it is connected and has no loops.

The graph $\Gamma$ is also assumed to have the following labelling by non-negative integers:
each arrowhead vertex $a\in\A$ is labelled by an integer $m_a$ called the \emph{multiplicity}.
On each edge $e\in\E$ connecting two vertices $v,w\in\A\cup\V$ there are
two positive integer weights $d_{ve}$ and $d_{we}$, the first one near $v$ (it is called the weight of $e$ \emph{adjacent} to $v$), the other one near $w$.
For any two edges $e$ and $e'$ adjacent to the same node $v$, the weights $d_{ve}$ and $d_{ve'}$ are assumed to be coprime.
A weight adjacent to an arrowhead or a leaf is always equal to $1$ (usually it is omitted when one draws a graph). An example
of a splice graph is presented in Figure~\ref{fig:graph}.

\smallskip
For a node $v$ we denote by $d_v$ the product of all weights $d_{ve}$ over all edges $e$ adjacent to $v$. If $v$ is an arrowhead or a leaf,
we define $d_v$ in a different way.
Let $e$ be the unique vertex adjacent to $v$ and let $w$ be its other end.
It is necessarily a node. Then $w$ is called the \emph{nearest node to $v$}, we denote it by $\vn$, and $d_{we}$ is the \emph{nearest weight to $v$}:
we shall denote it by $\dvn$. The \emph{weight} of $v$ is defined as $d_v:=d_w/d_{we}^2$. Note, that $d_v$ is not necessarily
an integer.

For any two vertices $v,w\in\A\cup\V$ we define the \emph{linking number} $\lk(v,w)$ as the product of all the weights adjacent
to, but not lying on, the shortest path connecting $v$ with $w$ in $\Gamma$, see \cite[page 84]{EN}. If $v$ is a node, we
define its \emph{multiplicity} $m_v$ as the sum
\[m_v=\sum_{a\in\A}\lk(a,v).\]
For example, for the graph in Figure~\ref{fig:graph} and the vertex $v_6$ the multiplicity is $4\cdot 3+4\cdot 3+4\cdot 2=32$.

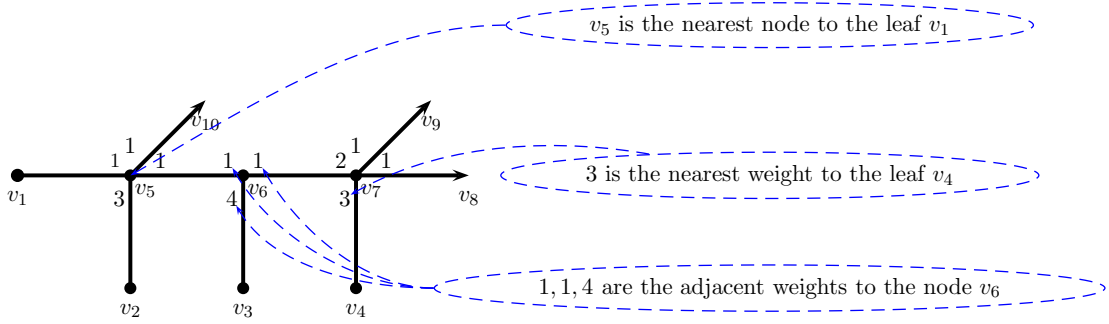
\begin{figure}[t]
\begin{pspicture}(0,-2)(8,3)
\psline[linewidth=1.5pt]{*->}(-3,0)(3,0)
\psline[linewidth=1.5pt]{->}(-1.5,0)(-0.5,1)
\psline[linewidth=1.5pt]{->}(1.5,0)(2.5,1)
\psline[linewidth=1.5pt](1.5,0)(1.5,-1.5)
\psline[linewidth=1.5pt](-1.5,0)(-1.5,-1.5)
\psline[linewidth=1.5pt](0,0)(0,-1.5)
\pscircle[fillstyle=solid,fillcolor=black](-3,0){0.08}
\pscircle[fillstyle=solid,fillcolor=black](-1.5,-1.5){0.08}
\pscircle[fillstyle=solid,fillcolor=black](-1.5,0){0.08}
\pscircle[fillstyle=solid,fillcolor=black](0,-1.5){0.08}
\pscircle[fillstyle=solid,fillcolor=black](0,0){0.08}
\pscircle[fillstyle=solid,fillcolor=black](1.5,0){0.08}
\pscircle[fillstyle=solid,fillcolor=black](1.5,-1.5){0.08}
\rput(-1.7,0.2){\rnode{W1}{\psscalebox{0.7}{$1$}}}
\rput(-1.1,0.2){\rnode{W2}{\psscalebox{0.8}{$1$}}}
\rput(-1.65,-0.3){\rnode{W3}{\psscalebox{0.8}{$3$}}}
\rput(-0.2,0.2){\rnode{W4}{\psscalebox{0.8}{$1$}}}
\rput(0.2,0.2){\rnode{W5}{\psscalebox{0.8}{$1$}}}
\rput(-0.15,-0.3){\rnode{W6}{\psscalebox{0.8}{$4$}}}
\rput(1.3,0.2){\rnode{W7}{\psscalebox{0.8}{$2$}}}
\rput(1.9,0.2){\rnode{W8}{\psscalebox{0.8}{$1$}}}
\rput(1.35,-0.3){\rnode{W9}{\psscalebox{0.8}{$3$}}}
\rput(-1.5,0.4){\rnode{W10}{\psscalebox{0.8}{$1$}}}
\rput(1.5,0.4){\rnode{W11}{\psscalebox{0.8}{$1$}}}
\rput(-3,-0.3){\psscalebox{0.8}{$v_1$}}
\rput(-1.5,-1.8){\psscalebox{0.8}{$v_2$}}
\rput(0,-1.8){\psscalebox{0.8}{$v_3$}}
\rput(1.5,-1.8){\psscalebox{0.8}{$v_4$}}
\rput(-1.3,-0.2){\psscalebox{0.8}{$v_5$}}
\rput(0.2,-0.2){\psscalebox{0.8}{$v_6$}}
\rput(1.7,-0.2){\psscalebox{0.8}{$v_7$}}
\rput(3,-0.3){\psscalebox{0.8}{$v_8$}}
\rput(2.5,0.7){\psscalebox{0.8}{$v_9$}}
\rput(-0.5,0.7){\psscalebox{0.8}{$v_{10}$}}
\rput(7,2){\ovalnode[linecolor=blue,linewidth=0.5pt,linestyle=dashed]{T1}{\psscalebox{0.8}{\textrm{$v_5$ is the nearest node to the leaf $v_1$}}}}
\rput(7,-1.5){\ovalnode[linecolor=blue,linewidth=0.5pt,linestyle=dashed]{T2}{\psscalebox{0.8}{\textrm{$1,1,4$ are the adjacent weights to the node $v_6$}}}}
\rput(7,0){\ovalnode[linecolor=blue,linewidth=0.5pt,linestyle=dashed]{T3}{\psscalebox{0.8}{\textrm{$3$ is the nearest weight to the leaf $v_4$}}}}
\rput(-1.5,0){\rnode{T4}{\hphantom{AA}}}
\nccurve[linestyle=dashed,angleA=180,linewidth=0.5pt,angleB=30,linecolor=blue]{->}{T1}{T4}
\nccurve[linestyle=dashed,angleA=180,linewidth=0.5pt,angleB=300,linecolor=blue]{->}{T2}{W4}
\nccurve[linestyle=dashed,angleA=180,linewidth=0.5pt,angleB=300,linecolor=blue]{->}{T2}{W5}
\nccurve[linestyle=dashed,angleA=180,linewidth=0.5pt,angleB=300,linecolor=blue]{->}{T2}{W6}
\nccurve[linestyle=dashed,angleA=170,linewidth=0.5pt,angleB=30,linecolor=blue]{->}{T3}{W9}
\end{pspicture}
\caption{An example of a graph link. We explain some terminology related to graph links.
The multiplicities of arrowheads are not presented.}\label{fig:graph}
\end{figure}

\subsection{Splicing and splice components}

There is one important procedure, namely the splicing of two graphs. It is easier to describe the inverse operation, which consists of cutting
an edge into two halves and changing them into arrowheads as in Figure~\ref{fig:splice}.

\begin{figure}[h]
\begin{pspicture}(-6,-1)(6,1.2)
\rput(-0.8,0){\psline(-6,-0.3)(-6,0.3)(-4.5,0.3)(-4.5,-0.3)(-6,-0.3)
\rput(-5.25,0){$\Gamma_1$}
\psline(-3.5,-0.3)(-3.5,0.3)(-2,0.3)(-2,-0.3)(-3.5,-0.3)
\rput(-2.75,0){$\Gamma_2$}
\rput(-6,0.5){\rnode{A}{}}
\rput(-2,0.5){\rnode{B}{}}
\psbrace(B)(A){\rotateright{$\Gamma$}}
\psline(-3.5,0)(-4.5,0)}
\psline[linewidth=1pt]{<-}(-2.3,0.1)(-0.2,0.1)\rput(-1.25,0.3){\psscalebox{0.8}{Splice}}
\psline[linewidth=1pt]{->}(-2.3,-0.1)(-0.2,-0.1)\rput(-1.25,-0.3){\psscalebox{0.8}{Cut}}
\rput(0.8,0){\psline(-0.5,-0.3)(-0.5,0.3)(1,0.3)(1,-0.3)(-0.5,-0.3)
\rput(0.25,0){$\Gamma_1$}
\psline(6.0,-0.3)(6.0,0.3)(4.5,0.3)(4.5,-0.3)(6.0,-0.3)
\rput(5.255,0){$\Gamma_2$}
\psline{->}(1,0)(2,0)\psline{->}(4.5,0)(3.5,0)
\rput(2.37,0.0){\psscalebox{0.8}{$(m_1)$}}
\rput(3.12,0.0){\psscalebox{0.8}{$(m_2)$}}}
\end{pspicture}
\caption{The graph $\Gamma$ on the left is a result of splicing of $\Gamma_1$ and $\Gamma_2$.}\label{fig:splice}
\end{figure}
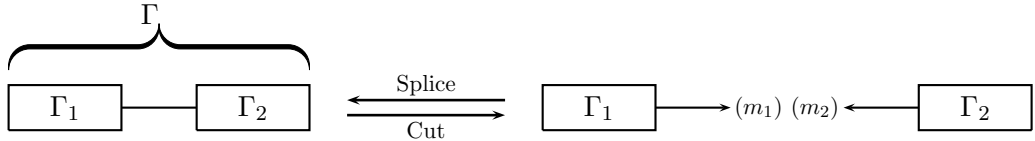

Here $m_1$ and $m_2$ are the multiplicities of the newly appeared arrowheads. They are uniquely determined by the condition that  if we cut the graph,
the multiplicities of all
nodes and leaves inside $\Gamma_1$ and $\Gamma_2$ are preserved. Splicing is the reverse procedure, it consists of taking two arrowheads of the two graphs
and joining them to form an edge connecting two graphs. In general, it is impossible to splice two graphs without
some conditions on the multiplicities of the arrowhead vertices $m_1$ and $m_2$. The whole procedure is described in details in \cite[pages 20--33]{EN}
or in \cite[Section 9]{NW}.

Given a graph $\Gamma$, we can decompose it as a union of so-called \emph{splice components}, where each splice component
contains exactly one node. A splice component is presented in Figure~\ref{fig:splicecomp}.

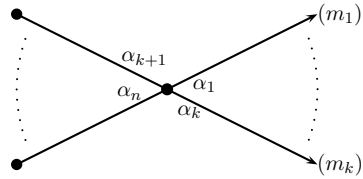
\begin{figure}[h]
\begin{pspicture}(-5,-2)(5,2)
\pscircle[fillcolor=black,fillstyle=solid](0,0){0.08}
\psline{->}(0,0)(2,1)\rput(2.3,1){\psscalebox{0.75}{$(m_1)$}}\rput(0.5,0.05){\psscalebox{0.75}{$\alpha_1$}}
\psline{->}(0,0)(2,-1)\rput(2.3,-1){\psscalebox{0.75}{$(m_k)$}}\rput(0.3,-0.3){\psscalebox{0.75}{$\alpha_k$}}
\psline(0,0)(-2,1)\pscircle[fillcolor=black,fillstyle=solid](-2,1){0.075}
\rput(-0.3,0.4){\psscalebox{0.75}{$\alpha_{k+1}$}}
\psline(0,0)(-2,-1)\pscircle[fillcolor=black,fillstyle=solid](-2,-1){0.075}
\rput(-0.5,-0.05){\psscalebox{0.75}{$\alpha_{n}$}}
\psarc[linestyle=dotted](0,0){2}{340}{20}
\psarc[linestyle=dotted](0,0){2}{160}{200}
\end{pspicture}
\caption{A splice component.}\label{fig:splicecomp}
\end{figure}

\subsection{Graphs and graph links}\label{ss:ggl}

In \cite{EN} it is shown that each graph such that the multiplicity of each arrowhead is $1$, gives rise to a link in a graph 3-manifold;
the arrowheads correspond to components of the link. If the multiplicity of some arrowheads are not all equal to $1$, we speak of a \emph{multilink},
see also \cite{EN}. There is a notion of a Seifert surface and a Seifert matrix for a multilink, the Tristram--Levine
signature for multilinks is defined by \eqref{eq:signature} and the average signature is as in \eqref{eq:average}.

Let us explain the connection of graphs and links in more detail. If $\alpha_1,\ldots,\alpha_n$
are pairwise coprime positive integers, the splice component in Figure~\ref{fig:splicecomp} represents a link in the Seifert homology sphere
$\Sigma:=\Sigma(\alpha_1,\ldots,\alpha_n)$. The link is the union of precisely $k$
fibers of the Seifert fibration $\Sigma\to S^2$, corresponding to $\alpha_1,\ldots,\alpha_k$.
In particular, if for some $j=1,\ldots,k$, $\alpha_j=1$, we take a non-singular fiber, otherwise we take a singular one.

Splicing two graphs corresponds to the following topological operation. Suppose we have two links $L=L_1\sqcup\ldots\sqcup L_m$ and $L'=L'_1\sqcup\ldots\sqcup L'_n$
in two integral homology spheres $M$ and $M'$ respectively. Consider $M\sm N(L_1)$ and $M'\sm N(L'_1)$ (here $N(K)$ is a tubular neighborhood of a knot $K$).
These are three manifolds whose boundary is a torus. We consider the link $L'':=L_2\sqcup\ldots\sqcup L_m\sqcup L'_2\sqcup\ldots\sqcup L'_n$ in the manifold
\[M''=(M\sm N(L_1))\#_f(M'\sm N(L'_1)),\]
where $f\colon S^1\times S^1\to S^1\times S^1$ exchanges the longitude and the meridian of $L_1$ and $L_1'$.
Then the pair $(M'',L'')$ is a result of splicing $(M,L)$ with $(M',L')$
along $L_1$ and $L_1'$. If $(M,L)$ is represented by a graph $\Gamma$ and $(M',L')$ by $\Gamma'$, then $(M'',L'')$ is represented by a splicing of $\Gamma$ and
$\Gamma'$ along arrowheads corresponding to $L_1$ and $L_1'$.

Let us make an important observation. If we forget about the link and look only at the underlying manifold $M$ (at the level of the graph it amounts to
replacing all the arrowheads by leaves), then the graph $\Gamma$ provides a JSJ decomposition of $M$.
Suppose $M\cong S^3$. Since a JSJ decomposition of the three sphere is trivial, all the splice components have to be
isomorphic to $S^3$. A Seifert fibred manifold $\Sigma(\alpha_1,\ldots,\alpha_n)$ is $S^3$ if and only if at most two of $\alpha_1,\ldots,\alpha_n$
are not equal to $1$. Therefore, if $M$ is a three sphere, the graph $\Gamma$ can not have any nodes with more than two adjacent weights different than $1$.
Throughout the paper we focus on the case when $M\cong S^3$, therefore from now on we shall make the following assumption about $\Gamma$.

\begin{assumption}\label{ass:ass2}
For any node at most 2 adjacent weights are different than $1$.
\end{assumption}
We point out that it is technically possible to give a formula for the average signature
in the general case, but one encounters additional problems in
sections~\ref{sec:graph2} and~\ref{sec:graph3}.

\subsection{The function $S(\Gamma)$}\label{ss:sgamma}

We shall define the function $S(\Gamma)$, which we shall use to compute the average signature of the underlying link.

\begin{definition}\label{def:SGamma}
For a graph $\Gamma$ we define
\[S(\Gamma)=\Sl(\Gamma)+\Sn(\Gamma)+\Sv(\Gamma)+\Se(\Gamma)+\Sa(\Gamma),\]
where
\begin{itemize}
\item[]{\bf (linking)}  The quantity $\Sl$ is twice the total linking number, that is
\[\Sl=\!\!\!\!\!\!\!\!\sum_{a,a'\in\A,a\neq a'}\lk(a,a').\]
\item[]{\bf (nodes)} The quantity $\Sn$ is the contribution of the nodes
\[\Sn=\!\!\!\!\!\!\!\!\sum_{v\in \V\colon\nu(v)>2}d_v(\nu(v)-2).\]
\item[]{\bf (leaves)} The quantity $\Sv$ comes from the leaves of $\Gamma$
\[\Sv=\!\!\!\!\!\!\!\!\sum_{v\in \V\colon\nu(v)=1}-d_v.\]
\item[]{\bf (edges)} The quantity $\Se$ is a sum of contributions of those edges that connect nodes. Let $e$ connect nodes $v$ and $w$ with
multiplicities $m_v$ and $m_w$. Suppose that upon cutting $\Gamma$ along $e$, the edge $e$ becomes two arrowheads with multiplicities $\mu_v$
and $\mu_w$. Set
$c=\gcd(\mu_v,\mu_w)$. The contribution of the edge $e$ to $\Se$ is equal to
\[
\begin{cases}
  c^2\left(\frac{d_{ve}}{\mu_vm_v}+\frac{d_{we}}{\mu_wm_w}-\frac{1}{\mu_v\mu_w}\right) & \textrm{if $\mu_v\mu_w\neq 0$}\\
  \frac{1}{d_v}-\frac{d_w}{d_{we}^2} & \textrm{if $\mu_w=0$}\\
  \frac{1}{d_w}-\frac{d_v}{d_{ve}^2} & \textrm{if $\mu_v=0$}.
\end{cases}
\]
\item[]{\bf (arrowheads)} The quantity $\Sa$ is a contribution of the arrowheads
\[\Sa=\!\!\sum_{a\in\A}\frac{d_a^\#}{m_{a^{\#}}},\]
where we recall that $a^{\#}$ is the nearest node to the arrowhead $a$ and $d_a^\#$ is the nearest weight to $a$.
\end{itemize}
\end{definition}

\goodbreak
\begin{remark}
  \begin{itemize}
\item[]
\item The formulae for $\Sn$ and $\Sv$ look similar and one could combine these two contribution in one term. However, the values of $d_v$
for $v$ a node and $v$ a leaf are different and computed in a different way.
\item There is a similarity between the function $S(\Gamma)$ and the function $W(\Gamma)$ defined in \cite[Section 4]{BZ}. The difference is equal to
$\Sl+\Se+\Sa$ and $\Sl$ has a clear topological meaning. We expect that the quantity $\Se+\Sa$
is small if $\Gamma$ is a graph of a link of a singularity. We know it is between $0$ and $2/9$, if $\Gamma$
is a link of a unibranched singular point, see \cite{Bo}.
\end{itemize}
\end{remark}

\begin{example}\label{ex:EN}
Let us consider the link in Figure~\ref{fig:EN} (it is taken from \cite[page 147]{EN}, only we changed the multiplicity of one arrowhead
vertex from $2$ to $1$). Both nodes of the graph are non-free. The quantity $S(\Gamma)$ is computed as follows.
\begin{itemize}
\item $\Sl=2\cdot(2\cdot 2\cdot 3)=24$. The $2$ in front comes from the fact that we compute the linking for each pair of arrowheads twice.
\item $\Sn=26\cdot(3-2)+6\cdot(4-2)=38$.
\item $\Sv=-\frac{6}{3^2}-\frac{6}{2^2}-\frac{13 \cdot 2}{2^2}=-\frac{2}{3}-\frac{3}{2}-\frac{13}{2}$.
\item To compute $\Se$ we observe that there is one edge connecting nodes. The multiplicities of the nodes are $M_v=38$ and $M_w=18$,
upon cutting the edge, the multiplicities of the two arrowheads are $m_v=6$ and $m_w=2$, hence $c=2$. We get
\[\Se=4\left(\frac{13}{38\cdot 6}+\frac{1}{18\cdot 2}-\frac{1}{6\cdot 2}\right)=\frac{1}{171}.\]
\item $\Sa$ is readily computed to be $\frac{1}{38}+\frac{1}{18}=\frac{14}{171}$.
\end{itemize}
We see that $S(\Gamma)=\frac{1015}{19}$.
\end{example}

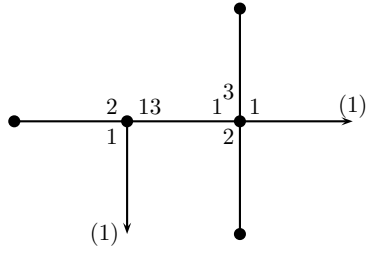
\begin{figure}
\begin{pspicture}(-5,-2)(5,2)
\psline{->}(-1.5,0)(3,0)
\psline(1.5,1.5)(1.5,-1.5)
\psline{->}(0,0)(0,-1.5)
\pscircle[fillstyle=solid,fillcolor=black](0,0){0.08}
\pscircle[fillstyle=solid,fillcolor=black](-1.5,0){0.08}
\pscircle[fillstyle=solid,fillcolor=black](1.5,0){0.08}
\pscircle[fillstyle=solid,fillcolor=black](1.5,1.5){0.08}
\pscircle[fillstyle=solid,fillcolor=black](1.5,-1.5){0.08}
\rput(-0.2,0.2){\psscalebox{0.8}{$2$}}
\rput(0.3,0.2){\psscalebox{0.8}{$13$}}
\rput(-0.2,-0.2){\psscalebox{0.8}{$1$}}
\rput(-0.3,-1.5){\psscalebox{0.8}{$(1)$}}
\rput(1.2,0.2){\psscalebox{0.8}{$1$}}
\rput(1.35,0.4){\psscalebox{0.8}{$3$}}
\rput(1.35,-0.2){\psscalebox{0.8}{$2$}}
\rput(1.7,0.2){\psscalebox{0.8}{$1$}}
\rput(3,0.2){\psscalebox{0.8}{$(1)$}}
\end{pspicture}
\caption{Splice graph from Example~\ref{ex:EN}.}\label{fig:EN}
\end{figure}

\subsection{Almost minimal diagrams}\label{ss:AMD}

The correspondence between graphs and graph links is not one to one. In fact, two graphs can give rise to the same link. For example, for a splice component
as in Figure~\ref{fig:splicecomp}, adding a leaf connected to the node, such that the near weight to the node is $1$, amounts to changing the Seifert
manifold from $\Sigma(\alpha_1,\ldots,\alpha_n)$ to $\Sigma(\alpha_1,\ldots,\alpha_n,1)$, the two manifolds are isomorphic and the links are the same.
This phenomenon is
discussed in detail in \cite[Section 8]{EN}, we shall need only a small portion of it.

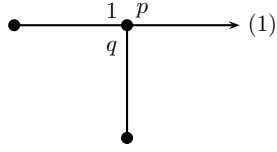
\begin{figure}
\begin{pspicture}(1,-2)(6,1)
\psline{->}(2,0)(5,0)
\psline(3.5,0)(3.5,-1.5)
\pscircle[fillstyle=solid,fillcolor=black](3.5,0){0.08}
\pscircle[fillstyle=solid,fillcolor=black](2,0){0.08}
\pscircle[fillstyle=solid,fillcolor=black](3.5,-1.5){0.08}
\rput(3.3,0.2){\psscalebox{0.8}{$1$}}
\rput(3.7,0.2){\psscalebox{0.8}{$p$}}
\rput(3.3,-0.3){\psscalebox{0.8}{$q$}}
\rput(5.3,0){\psscalebox{0.8}{$(1)$}}
\end{pspicture}
\caption{Non--trivial graph representing an unknot.}\label{fig:unknot}
\end{figure}

Observe that the diagram in Figure~\ref{fig:unknot} represents an unknot. To see this, note that this is a singular fiber of a fibration $S^3\to S^2$ given
by the restriction of the map $\pi\colon\mathbb{C}^2\sm\{0\}\to\C P^1$, $(z,w)\to [z^p\colon w^q]$. However $S(\Gamma)\neq 0$. In fact, $\Sl=0$,
$\Sn=q$, $\Sv=-p/q-pq$, $\Se=0$ and $\Sa=\frac{p}{q}$. In particular, Theorem~\ref{thm:main} does not hold for this kind of a diagram. As it will be clear
from the proof, the main reason is that there is a node such that the nearest weight is $1$. Therefore we shall need the following definition.

\goodbreak
\begin{definition}\
\begin{itemize}
\item[(a)]
Suppose $v$ is a leaf of a graph $\Gamma$ such that its nearest weight is $1$. Let $w$ be the nearest node and $\Gamma_w$ be the splice component of $\Gamma$
containing $w$. We shall call $v$ a \emph{bad leaf} if $\Gamma_w$ is not a graph $\Gamma(1,b)$ in Figure~\ref{fig:ab} (that is a graph $\Gamma(a,b)$ with $a=1$).
\item[(b)] A graph $\Gamma$ is called \emph{almost minimal} if it does not have any bad leaves.
\end{itemize}
\end{definition}

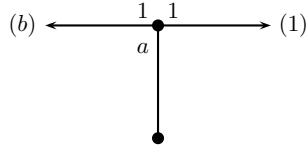
\begin{figure}
\begin{pspicture}(1,-2)(6,1)
\psline{<->}(2,0)(5,0)
\psline(3.5,0)(3.5,-1.5)
\pscircle[fillstyle=solid,fillcolor=black](3.5,0){0.08}
\pscircle[fillstyle=solid,fillcolor=black](3.5,-1.5){0.08}
\rput(3.3,0.2){\psscalebox{0.8}{$1$}}
\rput(3.7,0.2){\psscalebox{0.8}{$1$}}
\rput(3.3,-0.3){\psscalebox{0.8}{$a$}}
\rput(5.3,0){\psscalebox{0.8}{$(1)$}}
\rput(1.7,0){\psscalebox{0.8}{$(b)$}}
\end{pspicture}
\caption{A graph $\Gamma(a,b)$.}\label{fig:ab}
\end{figure}

Notice that by \cite[Theorem 8.1]{EN}, we can represent each graph link $L$ by an almost minimal diagram. In fact, using Move~3 from that result we can absorb
all the bad leaves.

\begin{lemma}\label{lem:stabil}
Suppose $\Gamma$ is a graph and $a$ is an arrowhead of multiplicity $1$. Let $\Gamma'$ be a graph resulting from $\Gamma$ by splicing this vertex with $\Gamma(1,b)$
for suitably defined $b$. Then $S(\Gamma)=S(\Gamma')$.
\end{lemma}

\begin{proof}
The proof of this result can be regarded as a warm-up before similar arguments in Sections~\ref{sec:inductstep} and~\ref{sec:inductstep2}.
We compare the values of $\Sl$, $\Sn$, $\Sv$, $\Se$ and $\Sa$ for the two graphs. The quantities $\Sl$ agree, $\Sn(\Gamma')-\Sn(\Gamma)=1$,
while $\Sv(\Gamma')-\Sv(\Gamma)=-1$. Let $d=d_a^{\#}$ be the nearest weight to the arrowhead $a$ in $\Gamma$ and $m$ the multiplicity of the
nearest node. We have $\Sa(\Gamma')-\Sa(\Gamma)=\frac{1}{b+1}-\frac{d}{m}$.
Furthermore $\Gamma'$ has one more edge connecting two nodes. Therefore $\Se(\Gamma')-\Se(\Gamma)=\frac{d}{m}+\frac{1}{b(b+1)}-\frac{1}{b}$
(notice that in the notation of Definition~\ref{def:SGamma} we have $\mu_v=1$, $\mu_w=b+1$ and $c=1$, if $v$ is the nearest node to $a$).

Summing up all the contributions we see that $S(\Gamma)=S(\Gamma')$.
\end{proof}

\section{Signatures of a graph link}
\subsection{Dedekind sums}
We begin with recalling the definition of the \emph{sawtooth function}:
\[
\sawtooth{x}=\begin{cases}
\{x\}-\frac12&x\not\in\Z\\
0&x\in\Z,
\end{cases}
\]
where $\{x\}$ denotes the fractional part. Given the above notation we introduce the Dedekind sums.

\begin{definition}
For two numbers $p,q$ such that $q>0$ we define the \emph{Dedekind sum} as
\[\s(p,q)=\sum_{j=0}^{q-1}\sawfrac{j}{q}\sawfrac{pj}{q}.\]
\end{definition}

We have several well known facts.
\begin{equation}\label{eq:cancel}
\begin{split}
\s(ap,aq)&=s(p,q)\textrm{\ \ \ for any $a\in\Z_{>0}$}\\
\s(p',q)&=s(p,q)\textrm{\ \ \ if $pp'=1\bmod q$}\\
\s(-p,q)&=-s(p,q)\\
\s(p+aq,q)&=s(p,q)\textrm{\ \ \ for any $a\in\Z$.}
\end{split}
\end{equation}

The most important relation we shall use is the Reciprocity Law. We refer to \cite{RG} for an excellent survey.

\begin{proposition}[Reciprocity Law]\label{prop:prp}
If $p,q$ are coprime, then
\begin{equation}\label{eq:dedrecip}
\s(p,q)+\s(q,p)=\frac{1}{12}\left(\frac{p}{q}+\frac{q}{p}+\frac{1}{pq}-3\right).
\end{equation}
\end{proposition}

Motivated by the above result we define

\begin{equation}\label{eq:Rs}
\Rs(p,q)=\frac{p}{q}+\frac{q}{p}+\frac{\gcd(p,q)^2}{pq}-3.
\end{equation}
By the first equation of \eqref{eq:cancel}, for all $p,q>0$, not necessarily coprime, \eqref{eq:dedrecip} translates into
\begin{equation}\label{eq:sandR}
\s(p,q)+\s(q,p)=\Rs(p,q)/12.
\end{equation}

We will also use the following generalization of \eqref{eq:dedrecip}.
\begin{proposition}[\expandafter{\cite[Theorem 7]{Pom}}]\label{prop:twotermlaw}
If $p,q,u,v$ are positive integers such that $\gcd(p,q)=\gcd(u,v)=1$ and $p',q'$ are such that $pp'+qq'=1$, then
\[
\s(p,q)+\s(u,v)=\s(p'u-q'v,t)-\frac{1}{4}+\frac{1}{12}\left(\frac{q}{vt}+\frac{v}{qt}+\frac{t}{qv}\right),\]
where $t=pv+qu$.
\end{proposition}

\subsection{Formulae for signatures involving Dedekind sums}\label{sec:algor}

Let $\Gamma$ be a splice graph and $L=L_\Gamma$ the corresponding graph (multi)link.
In this subsection we shall show how to compute the average signature of $L$ from the graph $\Gamma$. We write $\os(\Gamma)$ for $\os(L_\Gamma)$.

For $\lambda\in S^1$, let $\sigma_\lambda^-$ denote the equivariant signature of $L$. If $\lambda$ is of type $e^{2\pi ip/q}$, then
$\sigma_{\lambda}^-$ can be defined as $\sigma_\lambda(N_q,\tau^p)$ in \cite{Li} (where $N_q$ is as in Section~\ref{ss:motivation} and  $\tau$ is the deck
transformation), see also \cite[Section 12]{Gor}. For general $\lambda$ we refer to \cite[Appendix]{Neu0} or \cite{Mat}, but that definition is more
algebraic than topological.

We shall not need a definition of $\sigma_\lambda^-$. We will use only Lemma~\ref{lem:spladd} below, the fact that $\sigma_\lambda^-$ is zero except for
finitely many values of $\lambda$ (this follows from the algebraic description by \cite{Mat,Neu0}). For $\lambda\neq 1$, the signature
$\sigma_\lambda^-$ is equal to half the jump of the Tristram--Levine signature at $\lambda$, see \cite{Mat}. In particular,
if the Alexander polynomial $\Delta(t)$
of the link is not identically
zero, $\sigma_{\lambda}^-=0$ unless $\Delta(\lambda)=0$. For graph links, by the discussion of \cite[Chapter 11]{EN}, this implies that $\sigma_{\lambda}^-$ can
be non zero only for $\lambda$ of the form $e^{2\pi ip/q}$ with $p,q$ integers.

\begin{lemma}[see \expandafter{\cite[Theorem 5.3]{Neu2}}]\label{lem:spladd}\
\begin{itemize}
\item[(a)] The signatures $\sigma^-_\lambda$ are additive under splicing.
\item[(b)] For a splice component as in Figure~\ref{fig:splicecomp}
let us define
 $m_j=0$ for $j=k+1,\dots,n$. Let $\beta_j$ be chosen so that
\[\beta_j\alpha_1\ldots\widehat{\alpha_j}\ldots\alpha_n\equiv 1\ (\bmod\,\alpha_j).\]
Such $\beta_j$ exist because $\alpha_1,\ldots,\alpha_n$ are pairwise coprime. The multiplicity of the central vertex is equal to
\[m=\sum_{j=1}^n\alpha_1\ldots\widehat{\alpha_j}\ldots\alpha_nm_j.\]
Finally, let $s_j=(m_j-\beta_jm)/\alpha_j$ (it is easy to see that $s_j\in\mathbb{Z}$). If $\lambda=e^{2\pi ip/q}$ with
$p,q$ coprime, then
\begin{equation}\label{eq:spladd}
\sigma_\lambda^-=
\begin{cases}
0&\text{ if $q$ does not divide $m$}\\
2\sum_{j=1}^n\sawtooth{\frac{s_jp}{q}}&\text{ if $q$ divides $m$.}
\end{cases}
\end{equation}
\end{itemize}
\end{lemma}

We have a well-known lemma, for convenience of the reader we present a sketch of proof.
\begin{lemma}\label{lem:sig-equiv}
Let $L$ be a graph link and $\Gamma$ the underlying graph.
If $\zeta=e^{2\pi ix}$ is not a root of the Alexander polynomial of $L$ then the Tristram--Levine signature of $L$ is
related to the equivariant signatures by the following equation.
\begin{equation}\label{eq:sig-equiv}
\sigma(\zeta)=-\sum_{y>x}\sigma_{e^{2\pi i y}}^-+\sum_{y<x}\sigma_{e^{2\pi iy}}^-+1-\#\Gamma,
\end{equation}
where $\#\Gamma$ is the number of number of arrowheads of the graph (that is the number of components of $L$).
\end{lemma}
\begin{remark}
The result has in fact two parts. First is that
\begin{equation}\label{eq:lim1}
\sigma(\zeta)=-\sum_{y>x}\sigma_{e^{2\pi i y}}^-+\sum_{y<x}\sigma_{e^{2\pi iy}}^-+\lim_{\zeta\to 1}\sigma(\zeta).
\end{equation}
This is valid for all links, not only for the graph links. The other part, that $\lim_{\zeta\to 1}\sigma(\zeta)=1-\#\Gamma$ uses the fact
that the link is a graph link.
\end{remark}
\begin{proof}
The jumps of the Tristram Levine signatures are given by the equivariant signatures, compare \cite{Mat}.
This proves \eqref{eq:lim1}.
It remains to show that $\lim_{\zeta\to 1}\sigma(\zeta)=1-\#\Gamma$.

This appears to be a folklore result. Since we did not find a good reference in the literature, we present
a sketch of a proof, referring to \cite{EN,Neu1,Nem95,BoNe11} for some auxiliary results.

The restriction that the weights are positive implies by \cite[Theorem 11.2]{EN} that $L$ is fibered. Let $\Sigma$ be a fiber
and let $S$ be the corresponding Seifert matrix. It is non--degenerate.
A Jordan block decomposition of $S^{-1}S^T$ gives a decomposition $H_1(\Sigma;\R)=U_{\neq 1}\oplus U_{=1}$ such that $S$ has a block structure
with respect to this decomposition, $S=S_{\neq 1}\oplus S_{=1}$ and $S_{\neq 1}^{-1}S_{\neq 1}^T$ has eigenvalues different than $1$ and $S_{=1}^{-1}S_{=1}^T$
has eigenvalues equal to $1$. By \cite[Corollary 11.5]{EN} $S_{=1}^{-1}S_{=1}^T$ is actually the identity matrix.

The proof splits now into two parts. The first part is that for $\zeta$ sufficiently close to $1$ the signature $\sigma(\zeta)$ is equal to
the signature of the hermitian form $(1-\zeta)S_{=1}+(1-\ol{\zeta})S_{=1}^T$. This follows for example from \cite[Proposition 4.14]{BoNe11}
combined with the symmetry property of H-numbers \cite[Lemma 3.4(a)]{BoNe11}.

The second part is to show that the signature of the above hermitian form is $1-\#\Gamma$. Since $S-S^T=S(1-S^{-1}S^T)$, the intersection form
on $U_{\neq 1}$ is non--degenerate and on $U_{=1}$ it is zero. But this means that $U_{=1}$ is the image of $H_1(\p\Sigma)$ in $H_1(\Sigma)$,
in other words $U_{=1}$ is spanned by components $L_1,\ldots,L_n$ of $L$, subject to the relation $L_1+\ldots+L_n=0$.
It follows that $S_{=1}$ can be viewed as the linking matrix of $L_1,\ldots,L_{n-1}$, where we define $\lk(L_j,L_j)$ as $\lk(L_j,L_1+\ldots+L_{j-1}+L_{j+1}+L_n)$,
compare \cite[page 321]{Neu1}.
In particular $S_{=1}$ is symmetric and again by \cite[page 321]{Neu1} (see also \cite[Section 3.7]{BoNe12}) $S_{=1}$ is negative definite.
Thus the signature of $(1-\zeta)S_{=1}+(1-\ol{\zeta})S_{=1}^T$ is equal to $-(\#\Gamma-1)$ as desired.
\end{proof}

\begin{remark}\label{rem:sum}
If $L$ is a graph multilink such that each arrowhead has non-negative multiplicities, then \eqref{eq:sig-equiv} still holds
with the exception that $\#\Gamma$ should be understood as the number of arrowhead vertices \emph{with non-zero multiplicities}.
The argument is slightly more complicated since the Seifert matrix in general has to be decomposed into $S_0\oplus S_{\neq 1}\oplus S_{=1}$,
where $S_0$ is a zero matrix and $S_{\neq 1}$, $S_{=1}$ are as above. An argument very similar to the one in \cite[Section 3]{BoNe12} shows
that if the multilink has components $L_1,\ldots,L_n$ with multiplicities $m_1,\ldots,m_n$ and for some $k$, $m_j\ge 1$ if $j\le k$ and $m_j=0$ if $j>k$,
then $S_0$ has size $(m_1-1)+\ldots+(m_k-1)$ and $S_{=1}$ has size $k-1$ and is symmetric negative definite.
\end{remark}

It follows from Lemma~\ref{lem:sig-equiv} and Lemma~\ref{lem:spladd} that for a splice component as above, the Tristram--Levine
signature is given by the following formula:
\begin{equation}\label{eq:sigform2}
\sigma(e^{2\pi ix})=1-\#\Gamma+\sum_{j=1}^n\sum_{i=1}^m\sawfrac{is_j}{m}\delta_i,
\end{equation}
where we use notation from Lemma~\ref{lem:spladd} and $\delta_i=+1$ if $x>i/m$ and $-1$ if $x<i/m$.
The formula \eqref{eq:sigform2} holds as long as $mx$ is not an integer.

Integrating \eqref{eq:sigform2} over the interval $[0,1]$ yields the following result
\begin{equation}\label{eq:integralsplice}
\os(\Gamma)=1-\#\Gamma-4\sum_{j=1}^n \s(s_j,m),
\end{equation}
which is essentially due to N\'emethi \cite[Corollary~4.2]{Nem}.

The equivariant signatures are splice invariant by Lemma~\ref{lem:spladd}(a). If we splice two graphs $\Gamma_1$ and $\Gamma_2$ to a graph $\Gamma$,
and either $m_1$ or $m_2$ is zero (notation as in Figure~\ref{fig:splice}), then $\#\Gamma=\#\Gamma_1+\#\Gamma_2$, compare Remark~\ref{rem:sum}.
If $m_1m_2\neq 0$, then $\#\Gamma=\#\Gamma_1+\#\Gamma_2-1$. This
gives the following
fact, which we state now for future reference.
\begin{lemma}[Splice additivity of signatures]\label{lem:spladds}
Let $\Gamma$ be a graph, let us cut it into two components $\Gamma_1$ and $\Gamma_2$. Then $\os(\Gamma)=\os(\Gamma_1)+\os(\Gamma_2)+\eta$,
where $\eta=1$ if both newly appeared arrowheads have non-zero multiplicities, otherwise $\eta=0$.
\end{lemma}

\subsection{The main result}
Now we can state the main result of this article.

\begin{theorem}\label{thm:main}
If $L$ is a graph link in $S^3$ other than the unknot or the Hopf link.  Let $\Gamma$ be an almost minimal graph representing it, then
\[\os(L)=-\frac13 S(\Gamma).\]
\end{theorem}

The proof of Theorem~\ref{thm:main} is given in Section~\ref{s:proofofmain}, now let us provide some examples.

\begin{example}\label{ex:knot}
Assume that $\Gamma$ is an iterated torus knot as in Figure~\ref{fig:knot}.
With the notation in Figure~\ref{fig:knot} we have
\begin{align*}
\Sl&=0\\
\Sn&=\sum_{j=1}^n p_jq_j\\
\Sv&=-\frac{p_1}{q_1}-\sum_{j=1}^n\frac{q_j}{p_j}\\
\Se&=\sum_{j=1}^{n-1}\frac{1}{p_jq_j}-\frac{p_2}{q_2}-\dots-\frac{p_n}{q_n}\\
\Sa&=\frac{1}{p_nq_n}
\end{align*}
Adding this up we obtain
\[S(\Gamma)=\sum(p_j-1/p_j)(q_j-1/q_j).\]
It is known, see for example \cite{Bo},
that $\os=-\frac13 S(\Gamma)$ in this case. In particular, we have verified Theorem~\ref{thm:main} for all iterated torus knots.
\end{example}

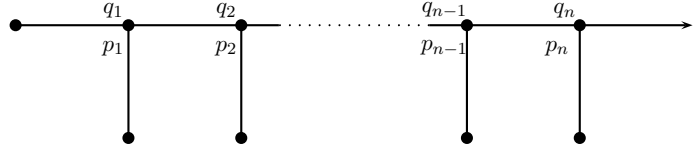
\begin{figure}[h]
\begin{pspicture}(-5,-2)(5,1)
\psline(-4.5,0)(-1,0)
\psline{->}(1,0)(4.5,0)
\psline[linestyle=dotted](-1,0)(1,0)
\psline(-3,0)(-3,-1.5)
\psline(-1.5,0)(-1.5,-1.5)
\psline(1.5,0)(1.5,-1.5)
\psline(3,0)(3,-1.5)
\pscircle[fillcolor=black,fillstyle=solid](-4.5,0){0.08}
\pscircle[fillcolor=black,fillstyle=solid](-3,0){0.08}
\pscircle[fillcolor=black,fillstyle=solid](-1.5,0){0.08}
\pscircle[fillcolor=black,fillstyle=solid](-3,-1.5){0.08}
\pscircle[fillcolor=black,fillstyle=solid](-1.5,-1.5){0.08}
\pscircle[fillcolor=black,fillstyle=solid](1.5,0){0.08}
\pscircle[fillcolor=black,fillstyle=solid](3,0){0.08}
\pscircle[fillcolor=black,fillstyle=solid](1.5,-1.5){0.08}
\pscircle[fillcolor=black,fillstyle=solid](3,-1.5){0.08}
\rput(-3.2,0.2){\psscalebox{0.8}{$q_1$}}
\rput(-1.7,0.2){\psscalebox{0.8}{$q_2$}}
\rput(1.2,0.2){\psscalebox{0.8}{$q_{n-1}$}}
\rput(2.8,0.2){\psscalebox{0.8}{$q_n$}}
\rput(-3.2,-0.3){\psscalebox{0.8}{$p_1$}}
\rput(-1.7,-0.3){\psscalebox{0.8}{$p_2$}}
\rput(1.2,-0.3){\psscalebox{0.8}{$p_{n-1}$}}
\rput(2.7,-0.3){\psscalebox{0.8}{$p_n$}}
\end{pspicture}
\caption{A graph $\Gamma$ representing an iterated torus knot.}\label{fig:knot}
\end{figure}

\begin{remark}
We remark that not all knots that are graph links are necessarily iterated torus knots. For example, all connected sums of iterated torus knots
are also graph links.
\end{remark}

\begin{example}
For the link from Example~\ref{ex:EN} we can compute the signature directly using~\eqref{eq:integralsplice}. The vertex on the left
contributes to
\[-1-4\left(\s(-19,38)+\s(-20,38)+\s(1,38)\right)=-1-4\cdot\frac{45}{19}.\]
The vertex of the right contributes to
\[-1-4\left(\s(1,18)+\s(-9,18)+\s(2,18)+\s(-12,18)\right)=-1-4\cdot\frac{11}{6}.\]
Hence the average signature is equal to $-\frac{1015}{57}$.
\end{example}

We point out that Theorem~\ref{thm:main} does not hold for general multilinks. As formula \eqref{eq:Gammaab} suggests, for general multilinks
one can not avoid Dedekind sums.

\section{Proof of Theorem~\ref{thm:main}}\label{s:proofofmain}

\subsection{Some terminology used in the proof}\label{ss:some}

We begin with the following definition, which is given to make precise notions, which are intuitively obvious.

\begin{definition}
A \emph{path} in a splice graph $\Gamma$ is a collection of nodes $v_1,\dots,v_k\in\V$ and edges $e_1,\dots,e_{k-1}\in\E$ such that $e_j$ connects $v_j$ to $v_{j+1}$
and all the nodes $v_1,\ldots,v_k$ are distinct. The \emph{length} of a path is the number of nodes occurring on the path.
The \emph{diameter} of a graph $\Gamma$, denoted $l(\Gamma)$, is the maximal length of a path in $\Gamma$.
\end{definition}

A graph with diameter $1$ has exactly one node. A graph with diameter $2$ has exactly two nodes. However, a graph with diameter $3$ can have arbitrary many nodes.
For instance, the graph in Figure~\ref{fig:graph2} has diameter~$3$.

\begin{definition}\label{def:elem}
For any positive integers $a,b$, the \emph{elementary} graph $\Gamma(a,b)$ is the graph as in Figure~\ref{fig:ab}.
\end{definition}

Applying \eqref{eq:integralsplice} to $\Gamma(a,b)$ we obtain a formula, which we will use several times in the future.

\begin{equation}\label{eq:Gammaab}
\os(\Gamma(a,b))=-1-\frac13\left(\Rs(1,a(b+1))+\Rs(b,(b+1)a)-\Rs(1,a)-12\s(a,b)\right).
\end{equation}

We shall use elementary graphs  to transform a graph of a multilink into a graph of a link in a controlled way.

\begin{definition}\label{def:complet}
Let $\Gamma$ be a graph with one arrowhead of multiplicity $m\neq 1$ and the multiplicity of any other arrowhead is equal to $1$.
The \emph{completion} of $\Gamma$
is the graph $\widetilde{\Gamma}$ obtained by
\begin{itemize}
\item If $m>1$: splicing a graph $\Gamma(m,b)$ to $\Gamma$ along the arrowhead with multiplicity $m$, where $b$
is a unique positive integer for which this is possible.
\item If $m=0$: replacing the arrowhead with multiplicity $1$ by a leaf.
\end{itemize}
If $\Gamma$ is a graph such that all the arrowheads have multiplicity $1$, we define the completion to be $\widetilde{\Gamma}=\Gamma$.
\end{definition}

We shall need one more notion.

\begin{definition}
A graph $\Gamma$ is \emph{simple} if all but possibly one of its splice components are
elementary.
\end{definition}
\subsection{Proof of Theorem~\ref{thm:main} up to a few lemmas}

We begin with the following lemma.
\begin{lemma}\label{lem:three}
Theorem~\ref{thm:main} holds for all simple graphs of diameter $3$, for which all the outer nodes (that is those that are adjacent
to one node) are elementary.
\end{lemma}
\begin{proof}
Let $\Gamma$ be a simple graph of diameter $3$ such that all the outer nodes are elementary. Then the only non-elementary node can be the central one.
Since the graph represents a link in $S^3$, at most two adjacent weights to the central node might be different than $1$, we denote them $p$ and $q$
(this also covers the case that $p=1$ or $q=1$). According to whether the edge of the central node is adjacent to a leaf or not, we have three cases
depicted in Figures~\ref{fig:graph2}, \ref{fig:graph3} and \ref{fig:graph4} (because $\Gamma$ is almost minimal).
We prove Theorem~\ref{thm:main} by a direct computation: in
Section~\ref{sec:graph2} we compute it for graphs in Figure~\ref{fig:graph2}. In Section~\ref{sec:graph3} we show, how the computations
change in the case of graphs in Figure~\ref{fig:graph3} and Figure~\ref{fig:graph4}.
\end{proof}

The main argument for passing from the special cases of low diameter to the general case, and in fact the gist of the proof of Theorem~\ref{thm:main}
is given by the following proposition.

\begin{proposition}\label{prop:inductstep}
Let $\Gamma$ be a graph of a link and $e$ an edge connecting two nodes. Let $\Gamma_1$ and $\Gamma_2$ be the graphs which are the result
of cutting $\Gamma$ along $e$ as in Figure~\ref{fig:splice}. Let $\wt{\Gamma}_1$ and $\wt{\Gamma}_2$ be
the completions along the arrowhead that appear as a result of cutting $e$.

If Theorem~\ref{thm:main} holds for $\wt{\Gamma}_1$ and $\wt{\Gamma}_2$, then it holds for $\Gamma$.
\end{proposition}
Proposition~\ref{prop:inductstep} is proved in Sections~\ref{sec:inductstep} and \ref{sec:inductstep2}.

In the following lemma we deal with graphs of small diameter.

\begin{lemma}\label{cor:two}
Theorem~\ref{thm:main} holds for all graphs of diameter $1$ and for all simple graphs of diameter $2$.
\end{lemma}
\begin{proof}

Let $\Gamma$ be a graph with a single node. Suppose it has only one arrowhead. Then, as the node can have at most two weights greater than one, the
valency of the node is at most three, so it has to be three. If $\alpha_1=1$ (we use notation in Figure~\ref{fig:splice}, here $k=1$ and $n=3$),
then the graph represents a torus knot. But Theorem~\ref{thm:main} has already been verified for torus knots; this follows
from the computation of $\os$ for torus knots in \cite{KM,Nem2,Bo} and is explained in Example~\ref{ex:knot}.
If $\alpha_1>1$, then $\alpha_2=1$ or $\alpha_3=1$. But then the graph is not almost minimal, so this case can not hold.

Assume that $\Gamma$ has one node but more than one arrowhead. We splice $\Gamma$ with two elementary diagrams of type $\Gamma(1,b)$
and obtain a simple graph of diameter 3, so
Lemma~\ref{lem:three} applies. Notice, that we might introduce a leaf with the nearest weight equal to $1$, but the graph is still almost minimal by construction.
Furthermore, by Lemma~\ref{lem:stabil} splicing $\Gamma$ does not affect $S(\Gamma)$. The topological type of the link does not change either.

If $\Gamma$ has diameter 2, it has two nodes $v_1$ and $v_2$. Suppose that the splice component corresponding to $v_2$ is elementary. If there
are no arrowhead vertices adjacent to $v_1$, it follows that $\Gamma$ is a splice diagram of an iterated torus knot, and the statement follows. If there is
an arrowhead vertex, we splice to it an elementary graph as above. We obtain a simple graph of diameter 3 such that the outer nodes are elementary, thus
we can use Lemma~\ref{lem:three} to conclude the proof in this case.
\end{proof}

\smallskip
\emph{Conclusion of the proof of Theorem~\ref{thm:main}}.

Consider a graph $\Gamma$ of diameter $l$ and suppose there are $k$ different paths of length $l$ on $\Gamma$. Suppose $l>3$ and let us
choose a path consisting of edges $e_1,\ldots,e_{l-1}$. We cut $\Gamma$ along $e_2$ and complete
the resulting graphs to $\wt{\Gamma}_1$ and $\wt{\Gamma}_2$.

It is clear from the construction that for $j=1,2$, $\wt{\Gamma}_j$ either has smaller diameter than $\Gamma$; or it has the same diameter, but lower
number of different paths of length $l$. An inductive step lets us reduce the proof of Theorem~\ref{thm:main} to the case of
graphs of diameter 3 and less.

Let us consider a graph of diameter 2. If it has two non-elementary nodes, we cut $\Gamma$ along the edge connecting the two nodes. The resulting
completed graphs $\wt{\Gamma}_1$ and $\wt{\Gamma}_2$ both have diameter 2 and are simple, therefore Lemma~\ref{cor:two} applies.

Let us now consider a graph $\Gamma$ of diameter $3$. If an outer edge $v$ of $\Gamma$ is not elementary, we cut $\Gamma$ along the unique
edge $e$ connecting $v$ to another node in $\Gamma$. We complete the two graphs to $\wt{\Gamma}_1$ that contains $v$, and $\wt{\Gamma}_2$. Then
$\wt{\Gamma}_1$ has two nodes and is simple. If $\wt{\Gamma}_2$ has diameter $2$, we conclude the proof. In general
$\wt{\Gamma}_2$ still might have diameter $3$, yet
it has one less non-elementary outer node. We
repeat the procedure for $\wt{\Gamma}_2$ until we arrive to the case,
when $\wt{\Gamma}_2$ has no non-elementary outer nodes and we conclude by Lemma~\ref{lem:three}.

\subsection{Theorem~\ref{thm:main} for the graph in Figure~\ref{fig:graph2}}\label{sec:graph2}
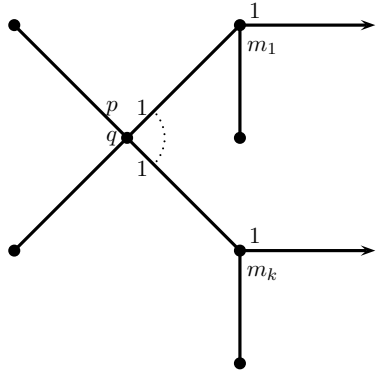
\begin{figure}
\begin{pspicture}(-6,-3)(6,2.5)
\pscircle[fillcolor=black,fillstyle=solid](0,0){0.08}
\pscircle[fillcolor=black,fillstyle=solid](1.5,1.5){0.08}
\pscircle[fillcolor=black,fillstyle=solid](1.5,-1.5){0.08}
\pscircle[fillcolor=black,fillstyle=solid](-1.5,1.5){0.08}
\pscircle[fillcolor=black,fillstyle=solid](-1.5,-1.5){0.08}
\pscircle[fillcolor=black,fillstyle=solid](1.5,0){0.08}
\pscircle[fillcolor=black,fillstyle=solid](1.5,-3){0.08}
\psline[linewidth=1.2pt](-1.5,-1.5)(1.5,1.5)
\psline[linewidth=1.2pt](1.5,-1.5)(-1.5,1.5)
\psline[linewidth=1.2pt](1.5,1.5)(1.5,0)
\psline[linewidth=1.2pt](1.5,-1.5)(1.5,-3)
\psline[linewidth=1.2pt,arrowsize=4pt]{->}(1.5,1.5)(3.3,1.5)
\psline[linewidth=1.2pt,arrowsize=4pt]{->}(1.5,-1.5)(3.3,-1.5)
\psarc[linestyle=dotted,dotsep=2pt](0,0){0.5}{-40}{40}
\rput(-0.2,0.4){\psscalebox{0.8}{$p$}}
\rput(-0.2,0){\psscalebox{0.8}{$q$}}
\rput(0.2,0.4){\psscalebox{0.8}{$1$}}
\rput(0.2,-0.4){\psscalebox{0.8}{$1$}}
\rput(1.8,1.2){\psscalebox{0.8}{$m_1$}}
\rput(1.8,-1.8){\psscalebox{0.8}{$m_k$}}
\rput(1.7,1.7){\psscalebox{0.8}{$1$}}
\rput(1.7,-1.3){\psscalebox{0.8}{$1$}}
\end{pspicture}
\caption{First of the three special graphs. The central vertex has $k$ edges. All the unmarked weights are
assumed to be $1$.}\label{fig:graph2}
\end{figure}

Let us begin with setting up a notational convention for Sections~\ref{sec:graph2} and \ref{sec:graph3}.

In Section~\ref{sec:graph2} $\Gamma$ denotes a splice graph as presented in Figure~\ref{fig:graph2}, in Section~\ref{sec:graph3} $\Gamma$
is the graph in Figure~\ref{fig:graph3} or in Figure~\ref{fig:graph4}.
The splice component of $\Gamma$ will be denoted by subscripts. The central splice components will be $\Gamma_{cen}$
and the components with multiplicities $m_1,\ldots,m_k$ shall be denoted by $\Gamma_1,\ldots,\Gamma_k$. The splice component adjacent
to the edge with weight $p$ shall be denoted $\Gamma_{\textrm{p}}$ (it is absent in Figure~\ref{fig:graph2}), and the component
adjacent to the edge with weight $q$ shall be denoted $\Gamma_{\textrm{q}}$: the latter one is present only in Figure~\ref{fig:graph4}. To distinguish
the last two components from $\Gamma_j$ for $j=p$ or $j=q$, we use roman fonts as subscripts.

The components $\Gamma_j$, $j=1,\ldots,k$, $\Gp$ and $\Gq$ are all elementary of type $\Gamma(m_j,M_j)$, $\Gamma(\mp,\Mp)$ and $\Gamma(\mq,\Mq)$,
where $M_j,\Mp,\Mq$ will be computed.

\smallskip
Now we pass to the case of graph in Figure~\ref{fig:graph2}. We have
\[M_j=M-pqm_j,\]
where
\[M=pq\sum_{j=1}^km_j\]
is the multiplicity of $\Gamma_{cen}$.
Summing \eqref{eq:Gammaab} for $\Gamma(m_1,M_1)$ up to $\Gamma(m_k,M_k)$ we obtain the following quantity:
\begin{equation}\label{eq:totalother}
-k-\frac13\sum_{j=1}^k\left(\Rs(1,m_j(M_j+1))+\Rs(M_j,(M_j+1)m_j)-\Rs(1,m_j)-12\s(m_j,M_j)\right).
\end{equation}
To compute $\os(\Gamma_{cen})$ we use the algorithm from Section~\ref{sec:algor}.
We have $\alpha_1=\dots=\alpha_k=1$, we introduce $\ap=p$ and $\aq=q$. Then $\beta_1=\dots=\beta_k=0$, $\bp=q'$ and $\bp=p'$, where $p'$ and
$q'$ are such that $pp'+qq'=1$. Then $s_j=m_j$ for $j=1,\dots,k$ and $\spp=-q'M/p$, $\sqq=-p'M/q$. Therefore $\os(\Gamma_{cen})$ is equal to
\[(1-k)-4\sum_{j=1}^k\s(m_j,M)-4\s(-q'M/p,M)-4\s(-p'M/q,M).\]
The latter formula can be transformed using \eqref{eq:cancel} to
\begin{equation}\label{eq:cen-vertex}
(1-k)-4\sum_{j=1}^k\left(\frac{1}{12}\Rs(m_j,M)-\s(M_j,m_j)\right)+4s(q',p)+4s(p',q).
\end{equation}
Combining \eqref{eq:totalother} and \eqref{eq:cen-vertex}
with splice additivity of $\os$ (Lemma~\ref{lem:spladds}), we obtain the following formula, which does not involve Dedekind sums anymore.

\begin{equation}\label{eq:os1}
\begin{split}
\os(\Gamma)=&1-k-\frac13\sum_{j=1}^k\left(\Rs(1,m_j(M_j+1))+\Rs(M_j,(M_j+1)m_j)+\right.\\
&+\left.\Rs(m_j,M)-\Rs(m_j,M_j)-\Rs(1,m_j)\right)+\frac13\Rs(p,q).
\end{split}
\end{equation}
We will now substitute \eqref{eq:Rs} for $\Rs$. As the formulae become more involved, to save the space we
present a formula for $-3\os(\Gamma)$. Here $c_j$ denotes $\gcd(m_j,M_j)=\gcd(m_j,M)$.
\begin{equation}\label{eq:sum}
\begin{split}
&\sum_{j=1}^k\left(
m_j(M_j+1)+\frac{2}{m_j(M_j+1)}+
m_j+\frac{m_j}{M_j}+\frac{M_j}{m_j(M_j+1)}+\frac{c_j^2}{M_j(M_j+1)m_j}+\right.\\ 
&+\left.\frac{m_j}{M}+\frac{M}{m_j}+\frac{c_j^2}{m_jM} 
-\frac{m_j}{M_j}-\frac{M_j}{m_j}-\frac{c_j^2}{m_jM_j} 
-\frac{2}{m_j}-m_j\right)-\\ 
&-\frac{p}{q}-\frac{q}{p}-\frac{1}{pq}.
\end{split}
\end{equation}
There are some cancellations in the above formula. We have $\sum\frac{m_j}{M}-\frac{1}{pq}=0$, $\sum(\frac{M}{m_j}-\frac{M_j}{m_j})=kpq$
and $\frac{1}{m_j(M_j+1)}+\frac{M_j}{m_j(M_j+1)}-\frac{1}{m_j}=0$.
We observe that
\begin{itemize}
\item $\Sl=\sum m_jM_j$;
\item $\Sn=kpq+\sum m_j$;
\item $\Sv=-\sum\frac{1}{m_j}-\frac{p}{q}-\frac{q}{p}$;
\item $\Se=\sum c_j^2\left(\frac{1}{m_jM}+\frac{1}{m_jM_j(M_j+1)}-\frac{1}{m_jM_j}\right)$;
\item $\Sa=\sum\frac{1}{m_j(M_j+1)}$.
\end{itemize}

Hence we obtain
\[-3\os=\Sl+\Sn+\Sv+\Se+\Sa=S(\Gamma)\]
as expected. Theorem~\ref{thm:main} holds in that case.

\subsection{Graphs in Figures~\ref{fig:graph3} and \ref{fig:graph4}}\label{sec:graph3}
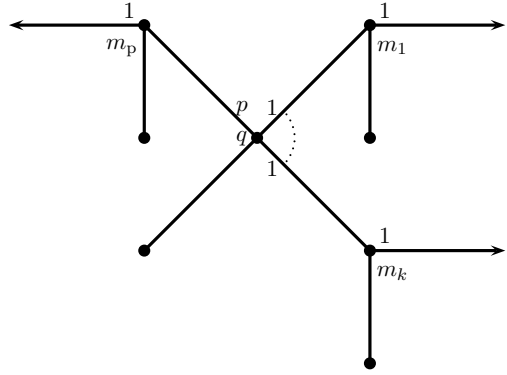
\begin{figure}
\begin{pspicture}(-6,-3)(6,2.5)
\pscircle[fillcolor=black,fillstyle=solid](0,0){0.08}
\pscircle[fillcolor=black,fillstyle=solid](1.5,1.5){0.08}
\pscircle[fillcolor=black,fillstyle=solid](1.5,-1.5){0.08}
\pscircle[fillcolor=black,fillstyle=solid](-1.5,1.5){0.08}
\pscircle[fillcolor=black,fillstyle=solid](-1.5,-1.5){0.08}
\pscircle[fillcolor=black,fillstyle=solid](1.5,0){0.08}
\pscircle[fillcolor=black,fillstyle=solid](-1.5,0){0.08}
\pscircle[fillcolor=black,fillstyle=solid](1.5,-3){0.08}
\psline[linewidth=1.2pt](-1.5,-1.5)(1.5,1.5)
\psline[linewidth=1.2pt](1.5,-1.5)(-1.5,1.5)
\psline[linewidth=1.2pt](1.5,1.5)(1.5,0)
\psline[linewidth=1.2pt](1.5,-1.5)(1.5,-3)
\psline[linewidth=1.2pt](-1.5,1.5)(-1.5,0)
\psline[linewidth=1.2pt,arrowsize=4pt]{->}(1.5,1.5)(3.3,1.5)
\psline[linewidth=1.2pt,arrowsize=4pt]{->}(1.5,-1.5)(3.3,-1.5)
\psline[linewidth=1.2pt,arrowsize=4pt]{->}(-1.5,1.5)(-3.3,1.5)
\psarc[linestyle=dotted,dotsep=2pt](0,0){0.5}{-40}{40}
\rput(-0.2,0.4){\psscalebox{0.8}{$p$}}
\rput(-0.2,0){\psscalebox{0.8}{$q$}}
\rput(0.2,0.4){\psscalebox{0.8}{$1$}}
\rput(0.2,-0.4){\psscalebox{0.8}{$1$}}
\rput(1.8,1.2){\psscalebox{0.8}{$m_1$}}
\rput(1.8,-1.8){\psscalebox{0.8}{$m_k$}}
\rput(1.7,1.7){\psscalebox{0.8}{$1$}}
\rput(-1.7,1.7){\psscalebox{0.8}{$1$}}
\rput(1.7,-1.3){\psscalebox{0.8}{$1$}}
\rput(-1.8,1.2){\psscalebox{0.8}{$\mp$}}
\end{pspicture}
\caption{Second of the special graphs. The $\os$ is computed in Section~\ref{sec:graph3}.}\label{fig:graph3}
\end{figure}

We shall now consider the graph in Figure~\ref{fig:graph3}. To compute $\os$ we will again sum over contributions of
all splice components. We shall denote $M=qm_p+pq\sum M_j$, and $M_j=M-pqm_j$. Furthermore let $\Mp=q\sum m_j$.
Then $M$ is the multiplicity of the central node and the elementary graphs $\Gamma_1,\ldots,\Gamma_k,\Gp$ are of type
$\Gamma(m_1,M_1),\ldots,\Gamma(m_k,M_k),\Gamma(\mp,\Mp)$ respectively. The value of $\os$ of these graphs is
given by \eqref{eq:Gammaab}.

To compute $\os(\Gamma_{cen})$
we choose $p'$ and $q'$ so that $pp'+qq'=1$. Then, by computations in Section~\ref{sec:algor} we obtain
\begin{equation}\label{eq:os-cen}
\os(\Gamma_{cen})=-k-4\left(\sum_j\s(m_j,M)+\s(-q'M/q,M)+\s((\mp-q'M)/p,M)\right).
\end{equation}
We recall that $M=p\Mp+q\mp$, so $(\mp-q'M)/p=p'\mp-q'\Mp$. Applying Proposition~\ref{prop:twotermlaw} we obtain
\begin{equation}\label{eq:magictrick}
\s((\mp-q'M)/p,M)=\s(p,q)+\s(\mp,\Mp)+\frac14-\frac{1}{12}\left(\frac{q\cp^2}{\Mp M}+\frac{\Mp}{qM}+\frac{M}{q\Mp}\right),
\end{equation}
where $\cp=\gcd(\mp,\Mp)$. Adding \eqref{eq:os-cen} to the sum of expressions from \eqref{eq:Gammaab} we obtain
\begin{multline*}
\os=\frac13\sum\left(\Rs(1,m_j(M_j+1))+\Rs(M_j,M_j(M_j+1))+\Rs(m_j,M)-\Rs(m_j,M_j)-\Rs(1,m_j)\right)+\\
+\frac13\left(\Rs(1,\mp(\Mp+1))+\Rs(\Mp,\mp(\Mp+1))-\Rs(1,\mp)\right)+\\
+1-\frac{1}{3}\left(\frac{q\cp^2}{\Mp M}+\frac{\Mp}{qM}+\frac{M}{q\Mp}\right)
\end{multline*}
Substituting again \eqref{eq:Rs} for $\Rs$ we obtain a formula for $-3\os(\Gamma)$, which differs from \eqref{eq:sum}
only in the last two lines.
\begin{equation}\label{eq:sum-2}
\begin{split}
&\sum_{j=1}^k\left(
m_j(M_j+1)+\frac{2}{m_j(M_j+1)}+
m_j+\frac{m_j}{M_j}+\frac{M_j}{m_j(M_j+1)}+\frac{c_j^2}{M_j(M_j+1)m_j}+\right.\\ 
&+\left.\frac{m_j}{M}+\frac{M}{m_j}+\frac{c_j^2}{m_jM} 
-\frac{m_j}{M_j}-\frac{M_j}{m_j}-\frac{c_j^2}{m_jM_j} 
-\frac{2}{m_j}-m_j\right)-\\ 
&+\mp(\Mp+1)+\frac{2}{\mp(\Mp+1)}+\frac{\mp}{\Mp(\mp+1)}+\frac{\cp^2}{\Mp(\Mp+1)\mp}+\frac{\mp}{\Mp}-\frac{2}{\mp}-\\
&-\frac{q\cp^2}{\Mp M}-\frac{\Mp}{qM}-\frac{M}{q\Mp}.
\end{split}
\end{equation}

The terms with $c_j$ and $\cp$ are equal to
\begin{equation}\label{eq:termswithc}
\sum\left(\frac{c_j^2}{M_j(M_j+1)m_j}+\frac{c_j^2}{m_jM}-\frac{c_j^2}{m_jM_j}\right)+\frac{\cp^2}{\Mp(\Mp+1)\mp}-\frac{q\cp^2}{\Mp M}.
\end{equation}
Observe that $\frac{p}{\mp M}-\frac{1}{\mp\Mp}=-\frac{q}{\Mp M}$,
 so \eqref{eq:termswithc} is actually equal to $\Se(\Gamma)$.
The terms $\sum(m_j/M)-(\Mp/qM)$ cancel, indeed $\Mp=q\sum m_j$. Then we can also simplify $\frac{\mp}{\Mp}-\frac{qM}{\Mp}=-\frac{p}{q}$.
Repeating the procedure from Section~\ref{sec:graph2} we conclude that $\os(\Gamma)=-\frac13S(\Gamma)$. We omit straightforward but
tedious details.

\bigskip
\begin{figure}
\begin{pspicture}(-6,-3)(6,2.5)
\pscircle[fillcolor=black,fillstyle=solid](0,0){0.08}
\pscircle[fillcolor=black,fillstyle=solid](1.5,1.5){0.08}
\pscircle[fillcolor=black,fillstyle=solid](1.5,-1.5){0.08}
\pscircle[fillcolor=black,fillstyle=solid](-1.5,1.5){0.08}
\pscircle[fillcolor=black,fillstyle=solid](-1.5,-1.5){0.08}
\pscircle[fillcolor=black,fillstyle=solid](1.5,0){0.08}
\pscircle[fillcolor=black,fillstyle=solid](-1.5,0){0.08}
\pscircle[fillcolor=black,fillstyle=solid](1.5,-3){0.08}
\pscircle[fillcolor=black,fillstyle=solid](-1.5,-3){0.08}
\psline[linewidth=1.2pt](-1.5,-1.5)(1.5,1.5)
\psline[linewidth=1.2pt](1.5,-1.5)(-1.5,1.5)
\psline[linewidth=1.2pt](1.5,1.5)(1.5,0)
\psline[linewidth=1.2pt](1.5,-1.5)(1.5,-3)
\psline[linewidth=1.2pt](-1.5,-1.5)(-1.5,-3)
\psline[linewidth=1.2pt](-1.5,1.5)(-1.5,0)
\psline[linewidth=1.2pt,arrowsize=4pt]{->}(1.5,1.5)(3.3,1.5)
\psline[linewidth=1.2pt,arrowsize=4pt]{->}(1.5,-1.5)(3.3,-1.5)
\psline[linewidth=1.2pt,arrowsize=4pt]{->}(-1.5,1.5)(-3.3,1.5)
\psline[linewidth=1.2pt,arrowsize=4pt]{->}(-1.5,-1.5)(-3.3,-1.5)
\psarc[linestyle=dotted,dotsep=2pt](0,0){0.5}{-40}{40}
\rput(-0.2,0.4){\psscalebox{0.8}{$p$}}
\rput(-0.2,0){\psscalebox{0.8}{$q$}}
\rput(0.2,0.4){\psscalebox{0.8}{$1$}}
\rput(0.2,-0.4){\psscalebox{0.8}{$1$}}
\rput(1.8,1.2){\psscalebox{0.8}{$m_1$}}
\rput(1.8,-1.8){\psscalebox{0.8}{$m_k$}}
\rput(1.7,1.7){\psscalebox{0.8}{$1$}}
\rput(-1.7,1.7){\psscalebox{0.8}{$1$}}
\rput(-1.7,-1.3){\psscalebox{0.8}{$1$}}
\rput(1.7,-1.3){\psscalebox{0.8}{$1$}}
\rput(-1.8,1.2){\psscalebox{0.8}{$\mp$}}
\rput(-1.8,-1.8){\psscalebox{0.8}{$\mq$}}
\end{pspicture}
\caption{The last of the three special graphs.}\label{fig:graph4}
\end{figure}
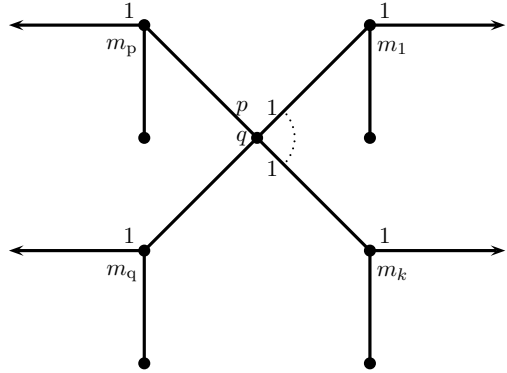

Essentially the same argument works for the graph in Figure~\ref{fig:graph4}. We need to use the trick from \eqref{eq:magictrick} twice, otherwise
the proof is essentially the same.

\subsection{Proof of Proposition~\ref{prop:inductstep}. First case}\label{sec:inductstep}
\begin{figure}[h]
\begin{pspicture}(-6,-4)(6,1.5)
\psline(-6,-0.3)(-6,0.3)(-4.5,0.3)(-4.5,-0.3)(-6,-0.3)
\rput(-5.25,0){$\Gamma_1$}
\psline(-3.5,-0.3)(-3.5,0.3)(-2,0.3)(-2,-0.3)(-3.5,-0.3)
\rput(-2.75,0){$\Gamma_2$}
\rput(-6,0.5){\rnode{A}{}}
\rput(-2,0.5){\rnode{B}{}}
\rput(3.5,0.5){\rnode{C}{}}
\rput(6.0,0.5){\rnode{D}{}}
\psbrace(B)(A){\rotateright{$\Gamma$}}
\psbrace(D)(C){\rotateright{$\wt{\Gamma}_2$}}
\psline[linewidth=1.5pt]{<-}(-1.8,0)(-0.7,0)\rput(-1.25,0.2){\psscalebox{0.8}{Splice}}
\psline(-3.5,0)(-4.5,0)
\psline(-0.5,-0.3)(-0.5,0.3)(1,0.3)(1,-0.3)(-0.5,-0.3)
\rput(0.25,0){$\Gamma_1$}
\psline(6.0,-0.3)(6.0,0.3)(4.5,0.3)(4.5,-0.3)(6.0,-0.3)
\rput(5.255,0){$\Gamma_2$}
\psline{->}(1,0)(2,0)\psline{->}(4.5,0)(3.5,0)
\rput(2.37,0.0){\psscalebox{0.8}{$(a)$}}
\rput(3.12,0.0){\psscalebox{0.8}{$(b)$}}
\rput(0,-1.7){
\psline(-6,-0.3)(-6,0.3)(-4.5,0.3)(-4.5,-0.3)(-6,-0.3)
\rput(-5.25,0){$\Gamma_1$}
\psline{->}(-4.5,0)(-1.5,0)\psline(-3,0)(-3,-1.5)
\pscircle[fillcolor=black,fillstyle=solid](-3,0){0.08}
\pscircle[fillcolor=black,fillstyle=solid](-3,-1.5){0.08}
\rput(-3.2,0.2){\psscalebox{0.8}{$1$}}
\rput(-3.2,-0.2){\psscalebox{0.8}{$a$}}
\rput(-2.8,0.2){\psscalebox{0.8}{$1$}}
\rput(-1.2,0){\psscalebox{0.8}{$(1)$}}
\rput(-6,-1.5){\rnode{E}{}}
\rput(-1.5,-1.5){\rnode{F}{}}
\psbrace(E)(F){\rotateleft{$\wt{\Gamma}_1$}}
}
\rput(0,-1.7){
\psline(6.0,-0.3)(6.0,0.3)(4.5,0.3)(4.5,-0.3)(6.0,-0.3)
\rput(5.255,0){$\Gamma_2$}
\psline{->}(4.5,0)(1.5,0)
\pscircle[fillcolor=black, fillstyle=solid](3,0){0.08}
\pscircle[fillcolor=black, fillstyle=solid](3,-1.5){0.08}
\psline{-}(3,0)(3,-1.5)
\rput(1.3,0){\psscalebox{0.8}{$(1)$}}
\rput(3.2,0.2){\psscalebox{0.8}{1}}
\rput(2.8,0.2){\psscalebox{0.8}{1}}
\rput(2.8,-0.2){\psscalebox{0.8}{$b$}}
\rput(1,-1.5){\rnode{G}{}}
\rput(6,-1.5){\rnode{H}{}}
\psbrace(G)(H){\rotateleft{$\wt{\Gamma}_2$}}
}
\end{pspicture}
\caption{Cutting graph, where the multiplicities of new arrows are greater than $0$.}\label{fig:completion}
\end{figure}
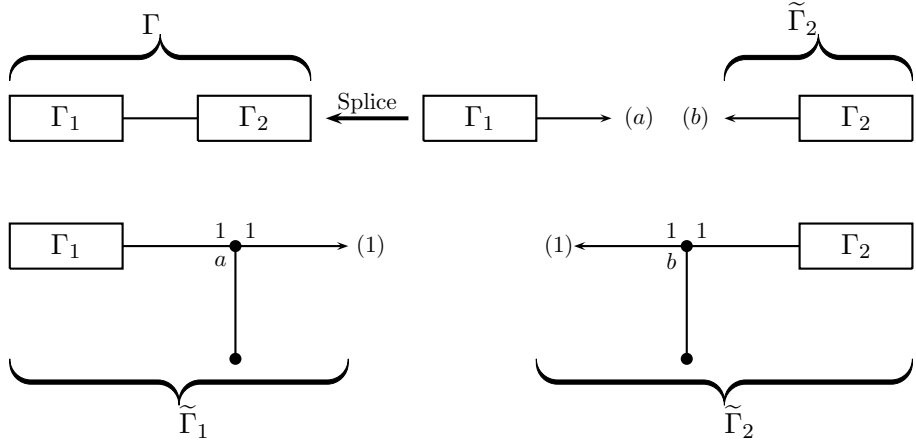

Let $a$ and $b$ be the multiplicities of arrowheads corresponding to the edge $e$ of the $\Gamma$, which we cut to obtain graphs $\Gamma_1$ and $\Gamma_2$.
 We suppose that $a$ is the multiplicity of the arrowhead belonging to graph $\Gamma_1$,
and $b$ the multiplicity of an arrowhead belonging to $\Gamma_2$. Let $v$ and $w$ be vertices of $\Gamma$ adjacent to the edge $e$, such that $v\in\Gamma_1$
and $w\in\Gamma_2$. We consider two cases. The first one is when
neither $a$ nor $b$ are equal to $0$. We will deal with the other case in Section~\ref{sec:inductstep2}.

\smallskip
If $ab\neq 0$, $\wt{\Gamma}_1$ is obtained from $\Gamma_1$ by splicing it with the graph $\Gamma(a,b)$ and $\wt{\Gamma}_2$ is
obtained from $\Gamma_2$ by splicing with the graph $\Gamma(b,a)$.

We have also to consider the case $a=1$ or $b=1$. If, say, $a=1$, according to Definition~\ref{def:complet}, the completion of $\Gamma_1$ is the same as $\Gamma_1$.
To deal with this case, we define $\wt{\Gamma_1}$ as $\Gamma_1$ with $\Gamma(1,b)$ spliced to it (by Lemma~\ref{lem:stabil} it does not affect $S(\Gamma_1)$,
so that this situation is still the same as in Figure~\ref{fig:completion}.
\begin{remark}
One can use the fact that $\Gamma$ is almost minimal to show that $a,b\neq 1$. However it is simpler treat the case $a=1$ equally with $a>1$.
\end{remark}

By splice additivity Lemma~\ref{lem:spladds} we have:
\begin{equation}\label{eq:once}
\os(\Gamma)+\os(\Gamma(a,b))+\os(\Gamma(b,a))+1=\os(\wt{\Gamma}_1)+\os(\wt{\Gamma}_2).
\end{equation}

We combine \eqref{eq:Gammaab} and the analogous expression for $\Gamma(b,a)$ to obtain.
\begin{align*}
-3-&\os(\Gamma(a,b))-3\os(\Gamma(b,a))=\\
=&\Rs(1,a(b+1))+\Rs(1,b(a+1)+\Rs(a,(a+1)b)+\Rs(b,(b+1)a)-\\
&-\Rs(a,b)-\Rs(1,a)-\Rs(1,b)+3=\\
=&2ab+a+b-\frac{1}{a}-\frac{1}{b}+\\
+&\frac{c^2}{ab(a+1)}+\frac{c^2}{ab(b+1)}-\frac{c^2}{ab}+\frac{1}{a(b+1)}+\frac{1}{b(a+1)},
\end{align*}
where $c=\gcd(a,b)$.

If Theorem~\ref{thm:main} holds for $\wt{\Gamma}_1$ and $\wt{\Gamma}_2$, then it holds for $\wt{\Gamma}$ if and only if
the last complicated expression is equal to $-S(\Gamma)+S(\wt{\Gamma}_1)+S(\wt{\Gamma}_2)$. Let us compare contributions to $S$. We
shall need the following result.

\begin{lemma}\label{lem:S4}
We have \emph{$\Sl(\Gamma)-\Sl(\wt{\Gamma}_1)-\Sl(\wt{\Gamma}_2)=-2ab$.}
\end{lemma}
\begin{proof}
Let $a^{1}_{new}$ and $a^{2}_{new}$ be the arrowheads of $\Gamma_1$, respectively of $\Gamma_2$, which appear as a result of cutting $\Gamma$ along
the edge $e$. Let  $a^{j}_1,\dots,a^{j}_{k_j}$ be the other arrowheads of $\Gamma_j$, $j=1,2$. These arrowheads might be regarded also
as arrowheads lying on $\Gamma$.

By the arguments in \cite[page 28]{EN} we have
\[a=\sum_{t=1}^{k_2}\lk_{\Gamma_2}(a^{2}_{new},a^2_{t}),\ \ \ b=\sum_{t=1}^{k_1}\lk_{\Gamma_1}(a^{1}_{new},a^1_{t}),\]
where the subscripts denote on which graph is the linking number computed.
Furthermore for any $r=1,\dots,k_1$ and any $t=1,\dots,k_2$ we have by \cite[Proposition 1.2]{EN}
\[\lk_\Gamma(a^1_r,a^2_t)=\lk_{\Gamma_1}(a^1_r,a^1_{new})\cdot\lk_{\Gamma_2}(a^2_{new}a^2_t).\]

Substituting the last two equations into the definition of $\Sl$, after straightforward computations we obtain the desired result.
\end{proof}

Given Lemma~\ref{lem:S4} we can compute the difference $S(\Gamma)-S(\Gamma_1)-S(\Gamma_2)$.
\begin{itemize}
\item $\Sl(\Gamma)-\Sl(\wt{\Gamma}_1)-\Sl(\wt{\Gamma}_2)=-2ab$ by Lemma~\ref{lem:S4};
\item $\Sn(\Gamma)-\Sn(\wt{\Gamma}_1)-\Sn(\wt{\Gamma}_2)=-a-b$ (the nodes from $\Gamma(a,b)$ and $\Gamma(b,a)$ contribute);
\item $\Sv(\Gamma)-\Sv(\wt{\Gamma}_1)-\Sv(\wt{\Gamma}_2)=\frac{1}{a}+\frac{1}{b}$ (there is a contribution of a single leaf
in $\Gamma(a,b)$ and $\Gamma(b,a)$);
\item $\Se(\Gamma)-\Se(\wt{\Gamma}_1)-\Se(\wt{\Gamma}_2)=-c^2(\frac{1}{ab(b+1)}+\frac{1}{ab(a+1)}-\frac{1}{ab})$,
where $c=\gcd(a,b)$. To explain this, observe that we cut an edge $e$ of $\Gamma$ and obtain two new edges: one in $\wt{\Gamma}_1$
and the other one in $\wt{\Gamma}_2$. The contribution of the edge on $\Gamma$ is $c^2(\frac{d_{ve}}{m_va}+\frac{d_{we}}{m_wb}-\frac{1}{ab})$,
The new edge on $\wt{\Gamma}_1$ contributes by $c^2(\frac{d_{ve}}{m_va}+\frac{1}{ab(b+1)}-\frac{1}{ab})$, the contribution of the new edge on
$\wt{\Gamma}_2$ is $c^2(\frac{d_{we}}{m_wb}+\frac{1}{ab(a+1)}-\frac{1}{ab})$.  The formula follows.
\item $\Sa(\Gamma)-\Sa(\wt{\Gamma}_1)-\Sa(\wt{\Gamma}_2)=-\frac{1}{a(b+1)}-\frac{1}{a(b+1)}$, in fact, $\wt{\Gamma}_1$ and $\wt{\Gamma}_2$
have exactly two new arrowheads as compared to $\Gamma$.
\end{itemize}

In particular we see that

\[S(\Gamma)-S(\wt{\Gamma}_1)-S(\wt{\Gamma}_2)=-3(\os(\Gamma(a,b))+\os(\Gamma(b,a))+1).\]

We use now this equation together with \eqref{eq:once}. Since $S(\wt{\Gamma}_j)=-3\os(\wt{\Gamma}_j)$ ($j=1,2$) by the assumption of the proposition,
the proof is finished.

\subsection{Proof of Proposition~\ref{prop:inductstep}. Second case}\label{sec:inductstep2}
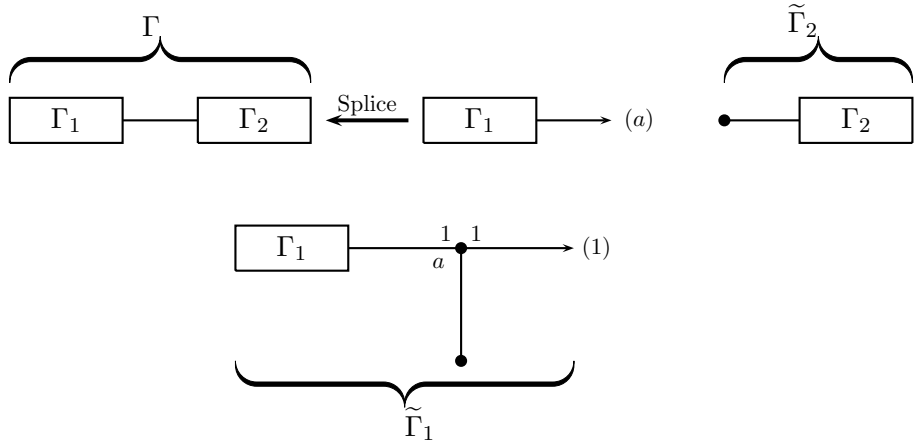
\begin{figure}[h]
\begin{pspicture}(-6,-4)(6,1.5)
\psline(-6,-0.3)(-6,0.3)(-4.5,0.3)(-4.5,-0.3)(-6,-0.3)
\rput(-5.25,0){$\Gamma_1$}
\psline(-3.5,-0.3)(-3.5,0.3)(-2,0.3)(-2,-0.3)(-3.5,-0.3)
\rput(-2.75,0){$\Gamma_2$}
\rput(-6,0.5){\rnode{A}{}}
\rput(-2,0.5){\rnode{B}{}}
\rput(3.5,0.5){\rnode{C}{}}
\rput(6.0,0.5){\rnode{D}{}}
\psbrace(B)(A){\rotateright{$\Gamma$}}
\psbrace(D)(C){\rotateright{$\wt{\Gamma}_2$}}
\psline[linewidth=1.5pt]{<-}(-1.8,0)(-0.7,0)\rput(-1.25,0.2){\psscalebox{0.8}{Splice}}
\psline(-3.5,0)(-4.5,0)
\psline(-0.5,-0.3)(-0.5,0.3)(1,0.3)(1,-0.3)(-0.5,-0.3)
\rput(0.25,0){$\Gamma_1$}
\psline(6.0,-0.3)(6.0,0.3)(4.5,0.3)(4.5,-0.3)(6.0,-0.3)
\rput(5.255,0){$\Gamma_2$}
\psline{->}(1,0)(2,0)\psline{-}(4.5,0)(3.5,0)
\pscircle[fillcolor=black,fillstyle=solid](3.5,0){0.08}
\rput(2.37,0.0){\psscalebox{0.8}{$(a)$}}
\rput(3,-1.7){
\psline(-6,-0.3)(-6,0.3)(-4.5,0.3)(-4.5,-0.3)(-6,-0.3)
\rput(-5.25,0){$\Gamma_1$}
\psline{->}(-4.5,0)(-1.5,0)\psline(-3,0)(-3,-1.5)
\pscircle[fillcolor=black,fillstyle=solid](-3,0){0.08}
\pscircle[fillcolor=black,fillstyle=solid](-3,-1.5){0.08}
\rput(-3.2,0.2){\psscalebox{0.8}{$1$}}
\rput(-3.3,-0.2){\psscalebox{0.8}{$a$}}
\rput(-2.8,0.2){\psscalebox{0.8}{$1$}}
\rput(-1.2,0){\psscalebox{0.8}{$(1)$}}
\rput(-6,-1.5){\rnode{E}{}}
\rput(-1.5,-1.5){\rnode{F}{}}
\psbrace(E)(F){\rotateleft{$\wt{\Gamma}_1$}}
}
\end{pspicture}
\caption{Cutting graph. One of the multiplicities is $0$.}\label{fig:below}
\end{figure}

We shall suppose that $b=0$. Then it cannot happen that $a=0$, for otherwise the graph has no arrowheads at all.

Observe that in that case $\os(\Gamma_2)=\os(\wt{\Gamma}_2)$.
Since $b=0$, we have $\os(\Gamma(a,b))=0$ (the graph represents an unknot), hence $\os(\Gamma_1)=\os(\wt{\Gamma}_1)$ by Lemma~\ref{lem:spladds} (notice
that $\eta=0$ in this case). In particular, we have
\[\os(\Gamma)=\os(\Gamma_1)+\os(\Gamma_2)=\os(\wt{\Gamma}_1)+\os(\wt{\Gamma}_2).\]

Now we look at the difference $S(\Gamma)-S(\wt{\Gamma}_1)-S(\wt{\Gamma}_2)$. Recall that the edge that is cut is denoted by $e$
and it connects vertices $v$ and $w$.

\begin{itemize}
\item $\Sl(\Gamma)-\Sl(\wt{\Gamma}_1)-\Sl(\wt{\Gamma}_2)=0$. The argument is as in Lemma~\ref{lem:S4}. This time $b=0$;
\item $\Sn(\Gamma)-\Sn(\wt{\Gamma}_1)-\Sn(\wt{\Gamma}_2)=-a$ (contribution of the vertex of $\Gamma(a,b)$);
\item $\Sv(\Gamma)-\Sv(\wt{\Gamma}_1)-\Sv(\wt{\Gamma}_2)=\frac{1}{a}+\frac{d_w}{d_{we}^2}$, the contribution is from the leaf of $\Gamma(a,b)$
and from the new leaf of $\wt{\Gamma}_2$.
\item $\Se(\Gamma)-\Se(\wt{\Gamma}_1)-\Se(\wt{\Gamma}_2)=\left(\frac{1}{d_v}-\frac{d_w}{d_{we}^2}\right)-\left(\frac{1}{d_w}-a\right)$. The first
expression is the contribution from the edge $e$ in $\Gamma$, the other comes from the new edge in $\wt{\Gamma}_1$ connecting $\Gamma_1$ to $\Gamma(a,b)$.
\item $\Sa(\Gamma)-\Sa(\wt{\Gamma}_1)-\Sa(\wt{\Gamma}_2)=-\frac{1}{a}$.
\end{itemize}

In particular we see that $S(\Gamma)=S(\wt{\Gamma}_1)+S(\wt{\Gamma}_2)$. By the induction assumption we conclude
the proof of  Proposition~\ref{prop:inductstep} as in the case $b\neq 0$ in Section~\ref{sec:inductstep}.

\end{document}